\documentclass[12pt,a4paper]{amsart}
\usepackage[utf8]{inputenc}
\usepackage{amsfonts}
\usepackage{amsmath, amssymb}
\usepackage{amsthm}
\usepackage{cases}
\usepackage[margin=.8in]{geometry}
\usepackage{tkz-euclide}
\usepackage{comment}
\usepackage{lineno} 

\newtheorem{theorem}{Theorem}[section]

\newtheorem{defin}[theorem]{Definition}
\newtheorem{lemma}[theorem]{Lemma}

\newtheorem{rem}[theorem]{Remark}

\usepackage{hyperref}
\hypersetup{
	colorlinks = true, 
	urlcolor = blue, 
	linkcolor = blue, 
	citecolor = blue 
}

\let \div \relax

\DeclareMathOperator{\div}{div}

\DeclareMathOperator{\Z}{\mathbb{Z}}
\DeclareMathOperator{\R}{\mathbb{R}}

\DeclareMathOperator{\Eu}{{\rm Eu}}
\DeclareMathOperator{\Sr}{{\rm Sr}}
\DeclareMathOperator{\Rey}{{\rm Re}}
\DeclareMathOperator{\Fr}{{\rm Fr}}
\DeclareMathOperator{\We}{{\rm We}}

\newcommand{\dd}{\, \mathrm{d}}
\newcommand{\del}{\partial}
\newcommand{\eps}{\varepsilon}
\newcommand{\weak}{\rightharpoonup}

\newcommand{\vc}{\mathbf}
\newcommand{\vu}{\vc u}
\newcommand{\vv}{\vc v}

\newcommand{\D}{\mathbf{D}}
\newcommand{\T}{\mathbf{T}}
\newcommand{\obst}{\mathcal{O}}
\newcommand{\oldr}[2]{\overset{\triangledown_{#1}}{#2}}
\newcommand{\Dv}[1]{\D(\vc{#1})}

\renewcommand{\S}{\mathbf{S}}

\renewcommand{\O}{\Omega}

\renewcommand{\rho}{\varrho}

\let\temp\phi
\let\phi\varphi
\let\varphi\temp

\allowdisplaybreaks

\title[Hom. reg. Oldroyd-type fluids]{Homogenization of regularized Oldroyd-type fluids}

\author{Florian Oschmann}
\address{Faculty of Mathematics and Physics, Charles University, Sokolovsk\'a 49/83, 186 00 Praha 8, Czech Republic.}
\email{florian.oschmann@matfyz.cuni.cz}

\author{Jonas Sauer}
\address{Institut f{\"u}r Mathematik, Friedrich-Schiller-Universit{\"a}t, Inselplatz 5, 07737 Jena, Germany}
\email{jonas.sauer@uni-jena.de}

\begin{document}

\begin{abstract}
    We study homogenization of a regularized viscoelastic Oldroyd-type model in a periodically perforated bounded domain. The system describes an incompressible non-Newtonian fluid coupled to an elastic extra stress tensor and includes both nonlinear viscosity and nonlinear stress diffusion effects.
    The governing model, introduced by Kreml, Pokorn\'y, and \v{S}alom (2015), covers Oldroyd-A- and Oldroyd-B-type constitutive laws.
    We establish qualitative and quantitative homogenization results in suitable scaling regimes and show convergence toward an effective Darcy law on the macroscopic domain.
    In particular, we prove that, under appropriate assumptions on the scaling parameters, the polymeric stress does not contribute to the effective limit equation.
    The analysis combines uniform estimates, oscillating test-function techniques, and a relative energy method, and additionally yields a weak--strong uniqueness principle for the viscoelastic system.
\end{abstract}

\maketitle

\section{Introduction}
Flows of complex fluids through porous media arise in a variety of applications, ranging from polymer transport and enhanced oil recovery to biological and industrial processes.
In many such situations, the microscopic geometry of the medium strongly influences the effective macroscopic behavior of the fluid.
Deriving effective large-scale equations from microscopic flow laws is therefore a central problem in homogenization theory.

Flows through porous media are frequently modeled by effective macroscopic equations such as Darcy’s law. A rigorous derivation of such effective models from microscopic fluid equations is a classical problem in homogenization theory. For Newtonian fluids, homogenization results in perforated domains are by now well understood and include both incompressible and compressible Navier–Stokes systems, see e.g.~\cite{Allaire1990a,Allaire1990b,LuYang2023,Tartar1980} and \cite{BasaricChaudhuri2024, DieningFeireislLu2017, HoferNecasovaOschmann2025, OschmannPokorny2023} and the references therein.
Recently, also certain classes of non-Newtonian flows have been understood, see e.g.~\cite{HoeferLuOschmann2025} and the references therein.

In contrast, homogenization results for \emph{viscoelastic} fluid models appear to be completely open.
The presence of elastic stresses and their nonlinear coupling with the velocity field introduces substantial additional analytical difficulties, particularly in perforated geometries.
In this work, we address for the first time qualitative and quantitative homogenization of viscoelastic fluids.
More precisely, we study a diffusive non-Newtonian Oldroyd-type model introduced by Kreml, Pokorn\'y, and \v{S}alom \cite{KremlPokornySalom2014}, and investigate its homogenization limit in a periodically perforated domain.

We show that, under suitable scalings, the microscopic dynamics converge to Darcy’s law, while the polymeric stress disappears in the effective macroscopic description.
This is consistent with the underlying asymptotic regime, which corresponds to vanishing inertial and elastic effects in the homogenization limit.
Our first main result is Theorem~\ref{thm:quali}, which establishes convergence of weak solutions toward Darcy's law in the homogenization limit.
We additionally derive quantitative convergence rates in Theorem~\ref{thm:quant} by means of a relative energy method. This relative energy approach enables us to additionally prove a weak--strong uniqueness principle for the scaled system in Theorem~\ref{thm:wk-str}.

The generalized Oldroyd model for a diffusive, time-dependent, non-Newtonian and viscoelastic fluid flow introduced by Kreml, Pokorn\'y, and \v{S}alom in \cite{KremlPokornySalom2014} reads as follows:
\begin{align}\label{NSE_not_scaled}
    \begin{cases}
        \div \vv = 0 & \text{in } (0,T) \times \O,\\
        \rho \del_t \vv+ \rho \div(\vv \otimes \vv) + \nabla \pi - \div(\mu( \lambda_{1} \Dv{v}) \Dv{v}) = \rho \vc f + \div \T & \text{in } (0,T) \times \O,\\
        \T + \zeta \oldr{}{\T} = 2 \eta \Dv{v} & \text{in } (0,T) \times \O,\\
        \vv = 0, \  \del_n \T = 0 & \text{on } (0,T) \times \del \O, \\
        \vv(0,\cdot) = \vv_{0}, \ \T(0,\cdot) = \T_{0} & \text{in } \O,
    \end{cases}
\end{align}
where $\Dv{v}=\frac12 (\nabla \vv + \nabla^T \vv)$ and ${\mathbf B}(\vv,\T) = {\mathbf W}\T-\T {\mathbf W} + a(\Dv{v} \T + \T \Dv{v})$ with ${\mathbf W}=\frac12 (\nabla \vv - \nabla^T \vv)$, and the Oldroyd derivative is given by
\begin{align*}
    \oldr{}{\T} = \partial_t \T + (\vv \cdot \nabla) \T - \eps_1 \div(\gamma(\lambda_2 \nabla\T)\nabla\T) - \vc B(\vv, \T).
\end{align*}

\medskip

In this system $\vv:(0,T)\times \O\to \R^3$ denotes the unknown velocity of the fluid, $\pi: (0,T)\times \O\to \R$ the unknown pressure, and $\T:(0,T)\times \O\to \R^{3\times 3}$ the unknown symmetric extra stress tensor,
while $\vv_0:\O\to\R^3$ and $\T_0:\O\to \R^{3\times 3}$ denote given initial data and $\vc f:(0,T)\times \O\to \R^3$ a given exterior forcing.
The function $\mu:\R^{3\times 3}\to \R$ introduces a nonlinear dependency of the dissipation on the symmetric gradient $\Dv{v}$, which makes the fluid non-Newtonian.
The term $\eps_1 \div(\gamma(\lambda_2 \nabla\T)\nabla\T)$ introduces a nonlinear stress diffusion, which regularizes the equation for the extra stress, with $\eps_1>0$ governing the size of this regularization, and $\gamma:\R^{3\times3\times3}\to\R$ determining its nonlinear nature.
The parameter $a\in [-1,1]$ determines the particular choice of Oldroyd viscoelasticity, with $a=-1$, $a=0$, and $a=1$ corresponding to the classical lower convected (Oldroyd A), corotational (Jaumann), and upper convected (Oldroyd B) models, respectively.
The parameter $\eta>0$ denotes the polymeric viscosity contribution, while $\lambda_1>0$ and $\lambda_2>0$ are scaling parameters associated with the nonlinear viscosity and stress diffusion laws, respectively. Lastly, $\zeta>0$ is the relaxation time of the fluid, and the parameter $\rho>0$ denotes the fluid's (constant) density.

\medskip

In this work, we consider qualitative and quantitative homogenization of this model.
More precisely, let $\Omega = \mathbb{T}^3$ be the three-dimensional torus. Let $Q = (-1,1)^3$ and $\obst \Subset Q$ be a reference particle, a closed simply connected smooth set with $0 \in {\rm int} \, \obst$. For $\eps > 0$ with $(2\eps)^{-1} \in \Z$, set
\begin{align*}
    \O_\eps = \O \setminus \bigcup_{k \in K_\eps} \eps (k + \obst), && K_\eps = \{k \in \Z^3 : \eps(k + \overline{Q}) \subset \O \}.
\end{align*}
On this level, the model \eqref{NSE_not_scaled}, if scaled properly in $\eps$ (see Appendix~\ref{sec:non-dim}), has the form
\begin{align}\label{NSE}
    \begin{cases}
        \div \vv_\eps = 0 & \text{in } (0,T) \times \O_\eps,\\
        \eps^\lambda \del_t \vv_\eps + N_\eps(\vv_\eps) + \nabla \pi_\eps  = \vc f + \eps^\xi \div \T_\eps & \text{in } (0,T) \times \O_\eps,\\
        \T_\eps + \eps^\kappa \oldr{}{\T}_\eps = 2 (1-\beta_\eps) \Dv{v_\eps} & \text{in } (0,T) \times \O_\eps,\\
        \vv_\eps = 0, \ \del_n \T_\eps = 0 & \text{on } (0,T) \times \del \O_\eps, \\
        \vv_\eps(0,\cdot) = \vv_{\eps 0}, \ \T_\eps(0,\cdot) = \T_{\eps 0} & \text{in } \O_\eps,
    \end{cases}
\end{align}
where
\begin{align*}
    N_\eps(\vv_\eps)&:=\eps^\lambda \div(\vv_\eps \otimes \vv_\eps) - \eps^2 \div(\mu(\We_1 \Dv{\vv_\eps})\Dv{\vv_\eps}),
\end{align*}
and where from the nondimensionalization we have introduced the quantities $\We_1>0$ and $\beta_\eps \in (0,1)$ in the viscosity contributions, and similarly the quantity $\Xi_\eps>0$ in the Oldroyd derivative
\begin{align*}
    \oldr{}{\T}_\eps = \partial_t \T_\eps + (\vv_\eps \cdot \nabla) \T_\eps - \eps_1 \div(\gamma(\Xi_\eps \nabla\T_\eps)\nabla\T_\eps) - \vc B(\vv_\eps, \T_\eps).
\end{align*}
In terms of characteristic quantities, this leads for the Strouhal, Euler, Froude, Reynolds, and Weissenberg number as well as for the viscosity ratio to
\begin{align}\label{Strouhaletc}
    \Sr=1 && \Eu=\eps^{-\lambda} && \Fr = \eps^\frac{\lambda}{2} && \Rey = \eps^{\lambda - \xi} && \We = \eps^\kappa && \beta = \eps^{2-\xi}.
\end{align}
As $\eps$ tends to $0$, these scalings lead to a singularly perturbed viscoelastic system in which the relative strength of inertia, viscosity, elasticity, and stress diffusion is controlled by the parameters $\lambda$, $\xi$, and $\kappa$.
Our analysis allows for a broad range of interactions between these effects.
The assumptions of Theorem~\ref{thm:quali} below identify a class of parameter regimes for which the microscopic dynamics converge to an effective Darcy law.
In particular, despite the presence of nonlinear viscoelastic stresses at the microscopic level, the polymeric stress tensor does not contribute to the effective macroscopic equation.
This shows that the Darcy limit is robust with respect to several different viscoelastic constitutive mechanisms and scaling choices.
We mention that since $\beta \in (0,1)$, it is natural to impose $\xi <2$.\\

In this work, the stress function $\gamma$ is assumed to satisfy the following conditions:
\begin{align}\label{assGamma}
\begin{cases}
    \gamma:\R^{3\times 3 \times 3} \to [0,\infty) \text{ is continuous},\\
    [\gamma(\varpi \nabla \T)\nabla\T - \gamma(\varpi \nabla\S)\nabla\S] : (\nabla\T - \nabla\S) \geq c \varpi^{q-2} |\nabla \T - \nabla \S|^q \\
    \quad \text{ for some } q \geq 4 \text{ and all symmetric } \T, \S \in \R^{3\times 3}, \, \varpi > 0.
\end{cases}
\end{align}

For the viscosity function, we employ the same assumptions as in \cite{HoeferLuOschmann2025}:
\begin{align}\label{assMu}
\begin{cases}
    \mu(\Dv{v}) = \mu_0 + g(|\Dv{v}|), \, \mu_0>0,\\
    g:[0,\infty) \to \R \text{ is continuous},\\
    |g(s)| \leq C s 1_{s \leq s_0} + C s^{\max\{p-2,0\}} 1_{s\geq s_0} \text{ for some } s_0>0, \, p>1,\\
    [\mu(\varpi \D_1)\D_1 - \mu(\varpi \D_2)\D_2] : (\D_1 - \D_2) \geq c |\D_1 - \D_2|^2 + c \varpi^{p-2} |\D_1 - \D_2|^p 1_{p>2} \\
    \quad \text{ for all symmetric } \D_1,\D_2 \in \R^{3\times 3}, \, \varpi > 0,
\end{cases}
\end{align}
where $1_{s \leq s_0}$ is the characteristic function of the set $\{s \leq s_0\} \subset [0,\infty)$.\\

At a formal level, one expects that, for suitable parameter regimes (for instance $\lambda > 1, \xi > 0$), the velocity field $\vv_\eps$ converges to $\vv$ as $\eps \to 0$, where $\vv$ is the solution to Darcy's law
\begin{align}\label{Darcy}
    \begin{cases}
    \div \vv = 0 & \text{in } (0,T) \times \O,\\
        \mu_0 M \vv = \vc f - \nabla \pi & \text{in } (0,T) \times \Omega .
    \end{cases}
\end{align}
Here, $M \in \R^{3 \times 3}$ is a symmetric positive definite matrix (called the \emph{resistance matrix}). For a formal derivation without the polymeric stress $\T$, we refer to \cite{HoferNecasovaOschmann2025}. We will show rigorously that the formal limit holds even for a wider range of $\xi$, depending on the values of $\kappa$, $\eps_1$, and $\Xi_\eps$.

Lastly, we will set $\Xi_\eps=1$ in the sequel:
By the coercivity assumption \eqref{assGamma}, the diffusion term contributes only through the combination $\eps_1\Xi_\eps^{q-2}$ in all \emph{a priori} estimates.
In turn, all the following results, including especially the qualitative estimates in Theorem~\ref{thm:quali}, hold true for arbitrary $\Xi_\eps>0$ when replacing $\eps_1$ by $\eps_1 \Xi_\eps^{q-2}$ there.

We emphasize that our results show that Darcy's law and especially the methods used to obtain it appear to be quite universal; in particular, we can cover several different models of viscoelastic flow while the limiting system stays the same for all of them. Moreover, our scaling in \eqref{NSE} corresponds to a time rescaling $t \mapsto t/\eps^\frac{\lambda}{2}$.
Hence, Darcy's law can be seen as a long-time asymptotic limit to system \eqref{NSE}.
In this context, we mention the recent work \cite{HNNP2026}, where the authors show that fluids of Oldroyd-B type behave almost Newtonian as time approaches infinity.\\

\paragraph{\bf Organization of the paper} In Section~\ref{sec:prelim}, we will collect useful lemmata needed in the sequel. In Section~\ref{sec:wkSol+result}, we define weak solutions to system \eqref{NSE} and state the main results. Sections~\ref{sec:unifEst} and \ref{sec:hom} are devoted to uniform estimates and the proof of the qualitative Theorem~\ref{thm:quali}, respectively.
In Section~\ref{sec:REI}, we derive a relative energy inequality and establish the quantitative convergence result in Theorem~\ref{thm:quant}.
Finally, in Section~\ref{sec:wk-str}, we prove the weak--strong uniqueness result Theorem~\ref{thm:wk-str}. \\

\paragraph{\bf Notation} If no ambiguity occurs, we will write $L^p L^q$ instead of $L^p(0,T; L^q(\Omega))$ or $L^p(0,T; L^q(\O_\eps))$. The symbol $a \lesssim b$ indicates that there is some constant $C>0$ which is independent of $a,b,$ and $\eps$ such that $a \leq C b$. Lastly, for a function $f$ defined on a domain $D \subset \R^3$ we denote its zero extension by $\tilde f$, that is,
\begin{align*}
    \tilde f = f \text{ on } D, && \tilde f = 0 \text{ on } \R^3 \setminus D.
\end{align*}

\section{Preliminaries}\label{sec:prelim}
In this section, we will provide lemmata that will be crucial for us in the sequel. The first one concerns Poincar\'e's inequality in the perforated domain $\O_\eps$:
\begin{lemma}\label{lem:Poinc}
For any $\vu \in W_0^{1,p}(\O_\eps; \R^3)$, we have
\begin{align*}
    \|\vu\|_{L^p(\O_\eps)} \lesssim \eps \|\nabla \vu\|_{L^p(\O_\eps)}.
\end{align*}
\end{lemma}
Next, we report the standard Korn's inequality:
\begin{lemma}\label{lem:Korn}
    For any $\vu \in W_0^{1,p}(\O_\eps; \R^3)$, we have
    \begin{align*}
        \|\nabla \vu\|_{L^p(\O_\eps)} \lesssim \|\Dv{u}\|_{L^p(\O_\eps)}.
    \end{align*}
\end{lemma}
Both results can be found in \cite{Allaire1989, Allaire1990a}.
We now introduce well-behaved test functions:
\begin{lemma}[see {\cite[Section~4.1]{HoferNecasovaOschmann2025}}] \label{lem:Weps}
    There exist matrix-valued functions $W_\eps \in W^{1,\infty}(\O; \R^{3 \times 3})$ and vector-valued functions $Q_\eps \in W^{1,\infty}(\O_\eps; \R^3)$ such that:
    \begin{enumerate}
        \item $\div W_\eps = 0$, $W_\eps = 0$ in $\O \setminus \O_\eps$, $\eps^2 (-\Delta W_\eps + \nabla Q_\eps) = M$, where $M \in \R^{3 \times 3}$ is a symmetric positive definite matrix (called the \emph{resistance matrix}\footnote{Note that in the notation of \cite{HoferNecasovaOschmann2025}, we have $M = \mathcal{K}^{-1}$, where $\mathcal{K}$ is the so-called \emph{permeability matrix}.}).
        \item We have
        \begin{align*}
            \|W_\eps\|_{L^\infty(\O)} + \eps \|\nabla W_\eps\|_{L^\infty(\O)} + \eps \|Q_\eps\|_{L^\infty(\O_\eps)} + \eps^2 \|\nabla Q_\eps\|_{L^\infty(\O_\eps)} \lesssim 1.
        \end{align*}
        \item We have
        \begin{align*}
            \|W_\eps - \mathbb{I}\|_{[W^{1,1}(\O)]'} \lesssim \eps .
        \end{align*}
    \end{enumerate}
\end{lemma}

The proof of our homogenization result rests on the idea to utilize a test function in the momentum equation of the form $W_\eps \phi$ with smooth $\phi$.
However, such function is in general not divergence-free.
To overcome this drawback, we recall a result from \cite[Lemma~2.3]{Hoefer2022}:
\begin{lemma}\label{lem:Bog-type}
    There exists an operator $B_\eps : W_0^{1,p}(\O; \R^3) \to W_0^{1,p}(\O_\eps; \R^3)$ for all $p \in [1,\infty)$ such that for all $\vu \in W_0^{1,p}(\O; \R^3)$ with $\div \vu = 0$, we have
    \begin{align*}
        \div B_\eps(\vu) = W_\eps : \nabla \vu, && \eps^{-1} \|B_\eps(\vu)\|_{L^p} + \|\nabla B_\eps(\vu)\|_{L^p} \lesssim \|(W_\eps - \mathbb{I}) : \nabla \vu\|_{L^p}.
    \end{align*}
\end{lemma}

\section{Weak solutions and main results}\label{sec:wkSol+result}
We begin with the definition of weak solutions to our system, existence of which was shown in \cite{KremlPokornySalom2014}:
\begin{defin}\label{def:wkSol}
    Let $D\subseteq \Omega$ be a domain with Lipschitz boundary, $\eps>0$, $\vv_0 \in L^2(D; \R^3)$ with $\div \vv_0 = 0$, $\T_0 \in L^2(D; \R^{3 \times 3})$ with $\T_0^\top = \T_0$, and $\vc f\in L^2((0,T)\times D; \R^3)$.
    Moreover, suppose
    \begin{align*}
        &\frac65 < p \leq 2, && q \geq 4, \text{ or} \\
        &2< p, && q > \frac{2p}{p-1}.
    \end{align*}
    Then, we call a couple $(\vv, \T)$ a \emph{weak solution} on $D$ to the system \eqref{NSE} if:
    \begin{itemize}
        \item The solution belongs to
        \begin{align*}
            &\vv \in L^\infty(0,T; L^2(D;\R^3)) \cap L^{\max\{2,p\}}(0,T; W_0^{1,\max\{2,p\}}(D;\R^3)) , \ \div \vv = 0, \\
            &\T \in L^\infty(0,T; L^2(D;\R^{3\times 3})) \cap L^q(0,T; W^{1,q}(D;\R^{3\times 3})), \ \del_n \T |_{\del D} = 0, \ \T^\top = \T, \\
            &\del_t \vv \in L^\sigma (0,T; [W_0^{1, \sigma'}(D;\R^3)]'), \ \text{for some } 1 \leq \sigma \leq \frac56 p, \\
            &\del_t \T \in L^{q'}(0,T; [W^{1,q}(D;\R^{3\times 3})]').
        \end{align*}

        \item For almost all $t \in (0,T)$ and all solenoidal $\phi \in C_c^\infty(D; \R^3)$, we have
        \begin{align*}
            &\eps^\lambda \langle \del_t \vv(t), \phi \rangle + \int_{D} \big( \eps^2 \mu(\We_1 \Dv{\vv}(t)) \Dv{\vv}(t) - \eps^\lambda \vv(t) \otimes \vv(t) \big) : \nabla \phi \dd x \\
            &= \int_{D} \vc f \cdot \phi + \eps^\xi \div \T(t) \cdot \phi \dd x .
        \end{align*}

        \item For almost all $t \in (0,T)$ and all symmetric tensor fields $\psi \in C^\infty(\overline{D};\R^{3\times 3})$, we have
        \begin{align*}
            &\eps^\kappa \langle \del_t \T(t), \psi \rangle + \int_{D} \eps^\kappa (\vv(t) \cdot \nabla) \T(t) : \psi \dd x \\
            &+ \int_{D} \eps_1 \eps^\kappa \gamma(\nabla \T(t)) \nabla \T(t) : \nabla \psi + \T : \psi \dd x \\
            &= 2 (1-\beta_\eps) \int_{D} \Dv{v}(t) : \psi \dd x - \eps^\kappa \int_{D} \vc B(\vv(t),\T(t)) : \psi \dd x .
        \end{align*}
        \item The initial data are attained in a weak sense, i.e., for any solenoidal $\phi \in C_c^\infty(D;\R^3)$, and any symmetric tensor field $\psi \in C^\infty(\overline{D};\R^{3\times 3})$, we have
        \begin{align*}
            \lim_{t \to 0} \int_{\O_\eps} \vv(t) \cdot \phi \dd x = \int_{\O_\eps} \vv_0 \cdot \phi \dd x, &&
            \lim_{t \to 0} \int_{\O_\eps} \T(t) : \psi \dd x = \int_{\O_\eps} \T_0 : \psi \dd x .
        \end{align*}
    \end{itemize}
\end{defin}
We note that in \cite{KremlPokornySalom2014}, the weak solution is defined with the term $\eps^\kappa \int_{D} \vc B(\vv(t),\T(t)) : \psi \dd x$ replaced by $- \eps^\kappa \int_{D} \vv(t) \cdot \div \vc A(\T(t), \psi) \dd x$, where $\vc A(\T, \psi) = \psi \T - \T \psi + a (\T \psi + \psi \T)$.
By integration by parts, both terms agree, and we have opted to keep the $\vc B$-term which resembles the original equation more closely.
We also note that by an approximation procedure, one might use $\psi=\T(t)$ itself as a test function, in which case $\vc A(\T,\T)=2a\T^2$ and hence
\begin{align}\label{BAconversion}
 \int_D \vc B(\vv(t),\T(t)):\T(t)\dd x = -2a\int_D \vv(t)\cdot \div \T^2(t) \dd x.   
\end{align}
Moreover, the regularity for $\vv$ obtained in \cite{KremlPokornySalom2014} is $\vv \in L^\infty L^2 \cap L^p W_0^{1,p}$, which does not imply our imposed regularity $\vv \in L^2 W_0^{1,2}$ if $p<2$. However, the form of $\mu$ in \eqref{assMu} enables us to repeat the proof of \cite{KremlPokornySalom2014} and obtain additionally $\vv \in L^2 W_0^{1,2}$.\\

Now we are in the position to formulate our main results.

\begin{theorem}[Qualitative convergence] \label{thm:quali}
    Suppose that we are given $\vc f\in L^2((0,T)\times\O;\R^3)$ and for each $\eps\in (0,1)$ initial data $(\vv_{\eps 0}, \T_{\eps 0}) \in L^2(\O_\eps;\R^3 \times \R^{3\times 3})$ satisfying uniformly in $\eps\in (0,1)$
    \begin{align}\label{init}
        \eps^\lambda \|\vv_{\eps 0}\|_{L^2}^2 + \eps^\kappa \|\T_{\eps 0}\|_{L^2}^2 \lesssim 1.
    \end{align}
    For $\eps\in (0,1)$, let $(\vv_\eps, \T_\eps)$ be a corresponding weak solution to system \eqref{NSE} on $\O_\eps$.
    Assume that
    \begin{align}\label{assCsts}
        \lambda > 1, && \We_1 / \eps \to 0, && \eps^{\xi(q-1) - \kappa} / \eps_1 \to 0, \ \xi, \kappa \in \R, \ \xi < 2.
    \end{align}
    Assume moreover that in the case $q\ne 4$ it holds
    \begin{align}\label{ass_eps_qn4}
        (\eps^{2(\kappa+\xi)(q-1)} + \eps^{(\kappa+\xi)(q-2)} ) / \eps_1^2 + ( \eps^{2(\kappa+\xi)(q-2)} + \eps^{(\kappa+\xi)(q-4) + 2q\xi} ) / \eps_1^4 \lesssim 1,
    \end{align}
    and in the case $q=4$ it holds
    \begin{align}\label{ass_eps_q4}
        (\eps^{6(\kappa+\xi)} + \eps^{2(\kappa+\xi)} ) / \eps_1^2 \lesssim 1, && (\eps^{\kappa+\xi} + \eps^{\kappa+3\xi} ) / \eps_1 \to 0.
    \end{align}
    Then there are $\vv\in L^2((0,T) \times \O;\R^3)$ and $\pi\in L^2((0,T);W^{1,2}(\O;\R))$ such that $\tilde{\vv}_\eps \weak \vv$ in $L^2((0,T) \times \O;\R^3)$, and $(\vv,\pi)$ solves Darcy's law \eqref{Darcy} with the resistance matrix $M$ given by Lemma~\ref{lem:Weps}.
\end{theorem}
\begin{rem}
    To determine which of the conditions in \eqref{ass_eps_qn4} is more restrictive, we compare the corresponding exponents. We have
    \begin{align*}
        2(\kappa+\xi)(q-1) \geq (\kappa+\xi)(q-2) \Leftrightarrow \kappa+\xi \geq 0,
    \end{align*}
    and
    \begin{align*}
        2(\kappa+\xi)(q-2) \geq (\kappa+\xi)(q-4) + 2q\xi \Leftrightarrow \kappa-\xi \geq 0 .
    \end{align*}
\end{rem}

\begin{theorem}[Quantitative convergence] \label{thm:quant}
    There exists $C>0$ such that the following holds.
    If in the situation of Theorem \ref{thm:quali} it holds $\vv \in W^{1,\infty}(0,T; L^2(\O;\R^3)) \cap L^2(0,T; W^{2,\infty}(\O;\R^3))$, then we have for all $\eps\in (0,1)$ that
    \begin{align*}
        \|\tilde{\vv}_\eps - \vv\|_{L^2L^2}^2 + \eps^\xi \|\T_\eps\|_{L^2L^2}^2 &\le  C\Big(\eps^\lambda \|\tilde{\vv}_{\eps 0} - \vv(0,\cdot)\|_{L^2}^2 + \eps^{\kappa+\xi} \|\T_{\eps 0}\|_{L^2}^2 + \eps^{2(\lambda-1)} + \eps^2 \\
        &\quad  + (\We_1 / \eps)^2 + [\eps^{\xi(q-1)-\kappa} / \eps_1]^\frac{1}{q-1} + \eps^{\kappa+\xi} / \eps_1^{\frac{1}{q-1}} + \eps^{\kappa+\xi} / \eps_1^{\frac{2}{q-2}} \Big).
    \end{align*}
\end{theorem}

\begin{rem}
    As can be seen from the above quantitative bound, for appropriate $\kappa, \xi, \lambda,$ and $\eps_1$, we have $\tilde \vv_\eps \to \vv$ strongly in $L^2((0,T) \times \Omega)$. In particular, one might let $\eps_1 \to \infty$, meaning that infinite stress dissipation leads to a strong vanishing effect of the polymeric stress.
\end{rem}

Lastly, we will show the following weak--strong uniqueness result for system \eqref{NSE}, which might be of independent interest. To this end, we forget about the dependence on $\eps$, that is, we set $\eps=1$.
\begin{theorem}\label{thm:wk-str}
    Let $\eps=1$, and let $D\subseteq \Omega$ be a domain with Lipschitz boundary.
    Assume that $(\vv, \T)$ is a weak solution on $D$ to system \eqref{NSE}, corresponding to the data $(\vv_{0}, \T_{0}) \in L^2(D;\R^3 \times \R^{3\times 3})$ and $\vc f\in L^2((0,T)\times D;\R^3)$.
    Let moreover
    $$(\vu, \S) \in W^{1,\infty}(0,T; L^2(D; \R^3\times \R^{3\times 3})) \cap L^2(0,T; W^{2,\infty}(D;  \R^3\times\R^{3\times 3}))$$
    be a strong solution on $D$ to \eqref{NSE} corresponding to the same data.
    Then $(\vv, \T) = (\vu, \S)$.
\end{theorem}

\begin{rem}
    Although our results are stated in the case of the torus $\Omega = \mathbb{T}^3$, they can be easily adapted to bounded smooth domains $\Omega \subset \R^3$, where the Darcy's law \eqref{Darcy} is completed with the boundary condition $\vv \cdot \vc n |_{\partial \Omega} = 0$. In particular, Theorems~\ref{thm:quali} and \ref{thm:wk-str} hold without change, whereas in Theorem~\ref{thm:quant}, the term $\eps^2$ needs to be replaced by $\eps$, which is due to a boundary layer corrector coming from the mismatch $\vv_\eps |_{\partial \Omega} = 0$ versus $\vv \cdot \vc n |_{\partial \Omega} = 0$. We refer to \cite{HoeferLuOschmann2025, HoferNecasovaOschmann2025} and leave the details to the interested reader.
\end{rem}

\section{Uniform estimates}\label{sec:unifEst}
In this section, we will deduce uniform in $\eps$ estimates from the energy inequalities \eqref{estV}, \eqref{estT}, and \eqref{TotalEn} for $(\vv_\eps, \T_\eps)$ listed in Appendix~\ref{sec:non-dim}.
More precisely, we show the following result.
\begin{lemma}\label{lem:unifBds}
    In the situation of Theorem \ref{thm:quali}, there is $\eps_*>0$ such that for all $\eps\in (0,\eps_*)$ there holds
\begin{align}
    &\|\vv_\eps\|_{L^2 L^2}^2 + \eps^\lambda \|\vv_\eps\|_{L^\infty L^2}^2 + \eps^2 \|\Dv{v_\eps}\|_{L^2 L^2}^2 + \eps^2 \We_1^{p-2} \|\Dv{v_\eps}\|_{L^p L^p}^p \mathbf{1}_{p>2} \notag \\
    &\quad + \eps^{\kappa+\xi} \|\T_\eps\|_{L^\infty L^2}^2 + \eps^{\kappa+\xi} \eps_1 \|\nabla \T_\eps\|_{L^q L^q}^q + \eps^\xi \|\T_\eps\|_{L^2 L^2}^2 \lesssim 1. \label{unifBds}
\end{align}
\end{lemma}
\begin{proof}
First, from \eqref{estV}--\eqref{estT}, we find the following energy inequalities: for almost every $\tau \in [0,T]$, we have
\begin{align*}
    &\left[ \int_{\O_\eps} \eps^\lambda \frac{1}{2} |\vv_\eps|^2 \dd x \right]_{t=0}^{t=\tau} + \int_0^\tau \int_{\O_\eps} \eps^2 \mu(\We_1 \Dv{v_\eps}) |\Dv{v_\eps}|^2 \dd x \dd t \\
    &\qquad \leq \int_0^\tau \int_{\O_\eps} \vv_\eps \cdot \vc f -  \eps^\xi \Dv{v_\eps} : \T_\eps \dd x \dd t, \\
    &\left[ \int_{\O_\eps} \eps^\kappa \frac{1}{2} |\T_\eps|^2 \dd x \right]_{t=0}^{t=\tau} + \int_0^\tau \int_{\O_\eps} \eps^\kappa \eps_1 \gamma( \nabla \T_\eps)|\nabla \T_\eps|^2 + |\T_\eps|^2 \dd x \dd t \\
    &\qquad = \int_0^\tau \int_{\O_\eps} 2 (1-\beta_\eps) \Dv{v_\eps} : \T_\eps - 2a \eps^\kappa \vv_\eps \cdot \div \T_\eps^2 \dd x \dd t .
\end{align*}
We also have from \eqref{TotalEn} that
\begin{align}\label{TotEnEps}
\begin{split}
    &\left[ \int_{\O_\eps} \eps^\lambda (1-\beta_\eps) |\vv_\eps|^2 + \frac{\eps^{\kappa+\xi}}{2} |\T_\eps|^2 \dd x \right]_{t=0}^{t=\tau} + \int_0^\tau \int_{\O_\eps} 2(1-\beta_\eps) \eps^2 \mu(\We_1 \Dv{v_\eps}) |\Dv{v_\eps}|^2 \dd x \dd t  \\
    &\qquad + \int_0^\tau \int_{\O_\eps} \eps^{\kappa+\xi} \eps_1 \gamma( \nabla \T_\eps) |\nabla \T_\eps|^2 + \eps^\xi |\T_\eps| ^2 \dd x \dd t \\
    &\leq 2\int_0^\tau \int_{\O_\eps} (1-\beta_\eps) \vv_\eps \cdot \vc f - a \eps^{\kappa+\xi} \vv_\eps \cdot \div \T_\eps^2 \dd x \dd t . 
\end{split}    
\end{align}

From the total energy inequality \eqref{TotEnEps}, we can extract uniform bounds for the functions $(\vv_\eps, \T_\eps)$ as follows. First, for any $\delta > 0$,
\begin{align*}
    \int_0^\tau \int_{\O_\eps} (1-\beta_\eps) \vv_\eps \cdot \vc f \dd x \dd t \lesssim \|\vv_\eps\|_{L^2 L^2} \lesssim \eps \|\Dv{v_\eps}\|_{L^2 L^2} \lesssim \delta \eps^2 \|\Dv{v_\eps}\|_{L^2 L^2}^2 + 1,
\end{align*}
such that, using coercivity properties of $\mu$ and $\gamma$ from \eqref{assGamma} and \eqref{assMu}, we see
\begin{align*}
    &\left[ \int_{\O_\eps} \eps^\lambda |\vv_\eps|^2 + \eps^{\kappa+\xi} |\T_\eps|^2 \dd x \right]_{t=0}^{t=\tau} + \int_0^\tau \int_{\O_\eps} \eps^2 |\Dv{v_\eps}|^2 + \eps^2 \We_1^{p-2} |\Dv{v_\eps}|^p \mathbf{1}_{p>2} \dd x \dd t \\
    &\qquad + \int_0^\tau \int_{\O_\eps} \eps^{\kappa+\xi} \eps_1 |\nabla \T_\eps|^q + \eps^\xi |\T_\eps|^2 \dd x \dd t \\
    &\lesssim 1 + \int_0^\tau \int_{\O_\eps} \eps^{\kappa+\xi} |\vv_\eps| | \T_\eps| |\nabla \T_\eps| \dd x \dd t.
\end{align*}
Now, we need the same trick as in \cite{KremlPokornySalom2014}: set $\T_M^\eps = |\Omega_\eps|^{-1} \int_{\O_\eps} \T_\eps \dd x$, integrate the equation \eqref{NSE}$_3$ for $\T_\eps$ over $\O_\eps$, and multiply with $\T_M^\eps$ to find
\begin{align*}
    \eps^\kappa \left[ \frac{|\Omega_\eps|}{2} |\T_M^\eps|^2 \right]_{t=0}^{t=\tau} + \int_0^\tau \int_{\O_\eps} |\T_M^\eps|^2 \dd x \dd t &= - \eps^\kappa \int_0^\tau \int_{\O_\eps} \vv_\eps \cdot \div (2a \T_\eps \T_M^\eps) \dd x \dd t \\
    &\lesssim \eps^\kappa \int_0^\tau |\T_M^\eps| \int_{\O_\eps} |\vv_\eps| |\nabla \T_\eps| \dd x \dd t.
\end{align*}
Hence, taking the supremum over $t \in (0,\tau)$, solving the resulting quadratic equation, and using the assumptions on the initial data \eqref{init}, we have
\begin{align*}
    \|\T_M^\eps\|_{L^\infty(0,\tau)} \lesssim 1 + \int_0^\tau \|\vv_\eps\|_{L^2} \|\nabla \T_\eps\|_{L^q} \dd t.
\end{align*}
In turn, using Lemma~\ref{lem:Poinc} and $q \geq 4$, we see that
\begin{align*}
    &\int_0^\tau \int_{\O_\eps} |\vv_\eps| |\T_\eps| |\nabla \T_\eps| \dd x \dd t \\
    &\lesssim \int_0^\tau \int_{\O_\eps} |\vv_\eps| |\T_\eps - \T_M^\eps| |\nabla \T_\eps| \dd x \dd t + \int_0^\tau |\T_M^\eps| \int_{\O_\eps} |\vv_\eps| |\nabla \T_\eps| \dd x \dd t \\
    &\lesssim \|\vv_\eps\|_{L^2 L^2} \|\nabla \T_\eps\|_{L^q L^q}^2 + \|\vv_\eps\|_{L^2 L^2} \|\nabla \T_\eps\|_{L^q L^q} + \|\vv_\eps\|_{L^2 L^2}^2 \|\nabla \T_\eps\|_{L^q L^q}^2 \\
    &\lesssim \eps \|\Dv{v_\eps}\|_{L^2 L^2} \|\nabla \T_\eps\|_{L^q L^q} + \eps \|\Dv{v_\eps}\|_{L^2 L^2} \|\nabla \T_\eps\|_{L^q L^q}^2 + \eps^2 \|\Dv{v_\eps}\|_{L^2 L^2}^2 \|\nabla \T_\eps\|_{L^q L^q}^2.
\end{align*}

Young's inequality now forces
\begin{align*}
    &\eps^{\kappa+\xi+1} \|\Dv{v_\eps}\|_{L^2 L^2} \|\nabla \T_\eps\|_{L^q L^q} \lesssim \delta \eps^2 \|\Dv{v_\eps}\|_{L^2 L^2}^2 + \eps^{2(\kappa+\xi)} \|\nabla \T_\eps\|_{L^q L^q}^2 \\
    &\lesssim \delta \eps^2 \|\Dv{v_\eps}\|_{L^2 L^2}^2 + (\eps^{2(\kappa+\xi)})^\frac{q-1}{q-2} \eps_1^{-\frac{2}{q-2}} + \delta \eps^{\kappa+\xi} \eps_1 \|\nabla \T_\eps\|_{L^q L^q}^q.
\end{align*}

Next, similarly,
\begin{align*}
    &\eps^{\kappa+\xi+1} \|\Dv{v_\eps}\|_{L^2 L^2} \|\nabla \T_\eps\|_{L^q L^q}^2 \lesssim \delta \eps^2 \|\Dv{v_\eps}\|_{L^2 L^2}^2 + \eps^{2(\kappa+\xi)} \|\nabla \T_\eps\|_{L^q L^q}^4 \\
    &\lesssim \delta \eps^2 \|\Dv{v_\eps}\|_{L^2 L^2}^2 + \chi_{q=4} (\eps^{2(\kappa+\xi)})^\frac{q-2}{q-4} \eps_1^{-\frac{4}{q-4}} + \delta \eps^{\kappa+\xi} \eps_1 \|\nabla \T_\eps\|_{L^q L^q}^q,
\end{align*}
where $\chi_{q=4} = 0$ if $q=4$ and $\chi_{q=4} = 1$ otherwise.
Indeed, in the case $q=4$ this follows since $\eps^{\kappa+\xi} / \eps_1 \ll 1$ for sufficiently small $\eps_*$ in light of assumption \eqref{ass_eps_q4}.

Thus,
\begin{align*}
    &\left[ \int_{\O_\eps} \eps^\lambda |\vv_\eps|^2 + \eps^{\kappa+\xi} |\T_\eps|^2 \dd x \right]_{t=0}^{t=\tau} + \int_0^\tau \int_{\O_\eps} \eps^2 |\Dv{v_\eps}|^2 + \eps^2 \We_1^{p-2} |\Dv{v_\eps}|^p \mathbf{1}_{p>2} \dd x \dd t \\
    &\qquad + \int_0^\tau \int_{\O_\eps} \eps^{\kappa+\xi} \eps_1 |\nabla \T_\eps|^q + \eps^\xi |\T_\eps|^2 \dd x \dd t \\
    &\lesssim 1 + (\eps^{2(\kappa+\xi)})^\frac{q-1}{q-2} \eps_1^{-\frac{2}{q-2}} \\
    &\qquad + \chi_{q=4} (\eps^{2(\kappa+\xi)})^\frac{q-2}{q-4} \eps_1^{-\frac{4}{q-4}} + \eps^{\kappa+\xi+2} \|\Dv{v_\eps}\|_{L^2 L^2}^2 \|\nabla \T_\eps\|_{L^q L^q}^2.
\end{align*}

To handle the last term, we come back to \eqref{estV}.
Using partial integration, we have
\begin{align*}
    \int_0^\tau \int_{\O_\eps} \vv_\eps \cdot \vc f - \eps^\xi \Dv{v_\eps} : \T_\eps \dd x \dd t &\lesssim \eps \|\Dv{v_\eps}\|_{L^2 L^2} (1 + \eps^\xi \|\nabla \T_\eps\|_{L^q L^q}) \\
    &\lesssim \delta \eps^2 \|\Dv{v_\eps}\|_{L^2 L^2}^2 + 1 + \eps^{2\xi} \|\nabla \T_\eps\|_{L^q L^q}^2.
\end{align*}
Using again \eqref{init}, this yields
\begin{align*}
    \eps^\lambda \|\vv_\eps\|_{L^\infty L^2}^2 + \eps^2 \|\Dv{v_\eps}\|_{L^2 L^2}^2 + \eps^2 \We_1^{p-2} \|\Dv{v_\eps}\|_{L^p L^p}^p \mathbf{1}_{p>2} \lesssim 1 
    + \eps^{2\xi} \|\nabla \T_\eps\|_{L^q L^q}^2,
\end{align*}
and hence
\begin{align*}
    &\eps^{\kappa+\xi+2} \|\Dv{v_\eps}\|_{L^2 L^2}^2 \|\nabla \T_\eps\|_{L^q L^q}^2 \lesssim \eps^{\kappa+\xi} ( 1 + \eps^{2\xi} \|\nabla \T_\eps\|_{L^q L^q}^2 ) \|\nabla \T_\eps\|_{L^q L^q}^2 \\
    &\lesssim \eps^{\kappa+\xi}
    \eps_1^{-\frac{2}{q-2}} + \delta \eps^{\kappa+\xi} \eps_1 \|\nabla \T_\eps\|_{L^q L^q}^q + \chi_{q=4} \eps^{\kappa+\xi} (\eps^{2\xi})^\frac{q}{q-4} \eps_1^{-\frac{4}{q-4}}.
\end{align*}
In total, we arrive at
\begin{align*}
    &\left[ \int_{\O_\eps} \eps^\lambda |\vv_\eps|^2 + \eps^{\kappa+\xi} |\T_\eps|^2 \dd x \right]_{t=0}^{t=\tau} + \int_0^\tau \int_{\O_\eps} \eps^2 |\Dv{v_\eps}|^2 + \eps^2 \We_1^{p-2} |\Dv{v_\eps}|^p \mathbf{1}_{p>2} \dd x \dd t \\
    &\qquad + \int_0^\tau \int_{\O_\eps} \eps^{\kappa+\xi} \eps_1 |\nabla \T_\eps|^q + \eps^\xi |\T_\eps|^2 \dd x \dd t \\
    &\lesssim 1 + (\eps^{2(\kappa+\xi)})^\frac{q-1}{q-2} \eps_1^{-\frac{2}{q-2}} + \chi_{q=4} (\eps^{2(\kappa+\xi)})^\frac{q-2}{q-4} \eps_1^{-\frac{4}{q-4}} 
    + \eps^{\kappa+\xi} 
    \eps_1^{-\frac{2}{q-2}} + \chi_{q=4} \eps^{\kappa+\xi} (\eps^{2\xi})^\frac{q}{q-4} \eps_1^{-\frac{4}{q-4}}
    \lesssim 1,
\end{align*}
where we have used the assumptions \eqref{ass_eps_qn4} for the case $q\ne 4$ and \eqref{ass_eps_q4} for the case $q=4$ in the last step.
Finally, utilizing Poincar\'e's and Korn's inequalities from Lemma~\ref{lem:Poinc} and \ref{lem:Korn}, we obtain \eqref{unifBds}.
\end{proof}

\section{Qualitative convergence}\label{sec:hom}
In this section, we will give the proof for Theorem \ref{thm:quali}.
From the uniform estimate \eqref{unifBds} obtained in Lemma \ref{lem:unifBds}, we infer that there exists $\vv\in L^2((0,T) \times \O;\R^3)$ such that
\begin{align*}
    \tilde{\vv}_\eps \weak \vv \text{ weakly in } L^2((0,T) \times \O;\R^3), \qquad \div\vv=0.
\end{align*}
We will show that $\vv$ is the solution to Darcy's law \eqref{Darcy}. We also recall the preliminary results from Section~\ref{sec:prelim} which will be used here several times.\\

To this end, for solenoidal $\phi \in C_c^\infty([0,T] \times \O;\R^3)$ and $B_\eps$ being the operator from Lemma~\ref{lem:Bog-type}, we can take $\vc w_\eps = W_\eps \phi - B_\eps(\phi)$ as a proper test function in the momentum equation \eqref{NSE}$_2$ to infer
\begin{align*}
    &\left[ \int_{\O_\eps} \eps^\lambda \vv_\eps \cdot \del_t \vc w_\eps \dd x \right]_{t=0}^{t=\tau} + \eps^\lambda \int_0^\tau \int_{\O_\eps} \vv_\eps \otimes \vv_\eps : \nabla \vc w_\eps \dd x \dd t \\
    &\qquad - \int_0^\tau \int_{\O_\eps} \eps^2 \mu(\We_1 \Dv{v_\eps}) \Dv{v_\eps} : \Dv{w_\eps} - \vc f \cdot \vc w_\eps - \eps^\xi \vc w_\eps \cdot \div \T_\eps \dd x \dd t = 0. 
\end{align*}

Note that thanks to the properties of $W_\eps$ and $B_\eps$, we have for any $p \in (1,\infty)$
\begin{align*}
    \|\vc w_\eps\|_{W^{1,\infty} L^p} + \eps \|\nabla \vc w_\eps\|_{W^{1,\infty} L^p} \leq C \|\phi\|_{W^{1,\infty} W^{1,p}}.
\end{align*}
We will now estimate all terms separately. First,
\begin{align*}
    &\left[ \int_{\O_\eps} \eps^\lambda \vv_\eps \cdot \del_t \vc w_\eps \dd x \right]_{t=0}^{t=\tau} \lesssim \eps^\lambda \|\vv_\eps\|_{L^\infty L^2} \|\vc w_\eps\|_{W^{1,\infty} L^2} \lesssim \eps^\frac{\lambda}{2},\\
    &\eps^\lambda \int_0^\tau \int_{\O_\eps} \vv_\eps \otimes \vv_\eps : \nabla \vc w_\eps \dd x \dd t \lesssim \eps^\lambda \|\vv_\eps\|_{L^2 L^\frac{6}{3-2\theta}}^2 \|\nabla \vc w_\eps\|_{L^\infty L^\frac{3}{2\theta}} \lesssim \eps^{\lambda - 1 - 2\theta} ,
\end{align*}
where we used the interpolation inequality
\begin{align*}
    \|\vv_\eps \otimes \vv_\eps\|_{L^\frac{3}{3-2\theta}} \leq \|\vv_\eps\|_{L^\frac{6}{3-2\theta}}^2 \leq \|\vv_\eps\|_{L^2}^{2(1-\theta)} \|\vv_\eps\|_{L^6}^{2\theta} \lesssim \eps^{-2\theta}, \ \theta \in (0,1).
\end{align*}
Note that $\lambda>1$ by \eqref{assCsts}, so that we may choose $\theta\in (0,1)$ with $\lambda-1-2\theta >0$.
Summarizing, we obtain
\begin{align*}
    &\left[ \int_{\O_\eps} \eps^\lambda \vv_\eps \cdot \del_t \vc w_\eps \dd x \right]_{t=0}^{t=\tau} + \eps^\lambda \int_0^\tau \int_{\O_\eps} \vv_\eps \otimes \vv_\eps : \nabla \vc w_\eps \dd x \dd t \to 0.
\end{align*}

Moreover, since $W_\eps \weak \mathbb{I}$ in $L^p$ for any finite $p$, which follows from $W_\eps \to \mathbb{I}$ strongly in $[W^{1,1}(\Omega)]'$ and the uniform bound of $W_\eps$ in $L^\infty$, and $\|B_\eps(\phi)\|_{L^p} \lesssim \eps \|\nabla \phi\|_{L^p}$, we infer $\vc w_\eps \weak \phi$ weakly in $L^p((0,T) \times \O)$ for any $p \in [1,\infty)$, which together with $\vc w_\eps = 0$ in $\O \setminus \O_\eps$ yields
\begin{align*}
    \int_0^\tau \int_{\O_\eps} \vc f \cdot \vc w_\eps \dd x \dd t \to \int_0^\tau \int_\O \vc f \cdot \phi \dd x \dd t.
\end{align*}
Furthermore, we may use assumption \eqref{assCsts} to obtain 
\begin{align*}
    \int_0^\tau \int_{\O_\eps} \eps^\xi \vc w_\eps \cdot \div \T_\eps \dd x \dd t \lesssim \eps^\xi \|\vc w_\eps\|_{L^\infty  L^2} \|\nabla \T_\eps\|_{L^q L^q} \lesssim \eps^\xi [\eps^{\kappa+\xi} \eps_1]^{-\frac1q} \to 0.
\end{align*}
Lastly, we focus on the dissipation term, which eventually will lead to the resistance term in Darcy's law \eqref{Darcy}.
Recall from \eqref{assMu} that we assumed
\begin{align*}
    \mu(\D) = \mu_0 + g(|\D|), && |g(s)| \lesssim s \mathbf{1}_{s \leq s_0} + s^{\max\{p-2, 0\}} \mathbf{1}_{s \geq s_0}, \ s_0 > 0.
\end{align*}
In turn, we can follow the presentation in \cite{HoeferLuOschmann2025}.
Indeed, we may rewrite
\begin{align}\label{dissip_decom_eq1}
\begin{split}
    &\int_0^\tau \int_{\O_\eps} \eps^2 \mu(\We_1 \Dv{v_\eps}) \Dv{v_\eps} : \Dv{w_\eps} \dd x \dd t \\
    &= \int_0^\tau \int_{\O_\eps} \eps^2 \mu_0 \Dv{v_\eps} : \Dv{w_\eps} \dd x \dd t + \int_0^\tau \int_{\O_\eps} \eps^2 g(\We_1 |\Dv{v_\eps}|) \Dv{v_\eps} : \Dv{w_\eps} \dd x \dd t \\
    &= \int_0^\tau \int_{\O_\eps} \eps^2 \mu_0 \Dv{v_\eps} : \D(W_\eps \phi) \dd x \dd t - \int_0^\tau \int_{\O_\eps} \eps^2 \mu_0 \Dv{v_\eps} : \D B_\eps(\phi) \dd x \dd t \\
    &\qquad + \int_0^\tau \int_{\O_\eps} \eps^2 g(\We_1 |\Dv{v_\eps}|) \Dv{v_\eps} : \Dv{w_\eps} \dd x \dd t .
\end{split}
\end{align}
Similarly to \cite{HoeferLuOschmann2025}, we may use $\div\vv_\eps=0$ to obtain
\begin{align*}
    &\int_0^\tau \int_{\O_\eps} \eps^2 \mu_0 \Dv{v_\eps} : \D(W_\eps \phi) \dd x \dd t \\
    &= \int_0^\tau \mu_0 \langle \eps^2 (-\Delta W_\eps + \nabla Q_\eps), \tilde{\vv}_\eps \otimes \phi \rangle \dd t + \eps^2 \int_0^\tau \int_{\O_\eps} \vc z_\eps \cdot \vv_\eps \dd x \dd t \\
    &= \int_0^\tau \int_{\O} \mu_0 M \tilde{\vv}_\eps \cdot \phi \dd x \dd t + \eps^2 \int_0^\tau \int_{\O_\eps} \vc z_\eps \cdot \vv_\eps \dd x \dd t,
\end{align*}
where $\vc z_\eps = (\Delta W_\eps) \phi - \Delta(W_\eps \phi) - (Q_\eps \cdot \nabla) \phi$.
In particular $\|\vc z_\eps\|_{L^2 L^2} \lesssim \eps^{-1}$, and thus
\begin{align*}
    \eps^2 \int_0^\tau \int_{\O_\eps} \vc z_\eps \cdot \vc v_\eps \dd x \dd t \lesssim \eps^2 \|\vc z_\eps\|_{L^2 L^2} \|\vv_\eps\|_{L^2 L^2} \lesssim \eps \to 0.
\end{align*}
Since $\tilde{\vv}_\eps \weak \vv$ weakly in $L^2((0,T) \times \O)$, we infer
\begin{align*}
    &\int_0^\tau \int_{\O_\eps} \eps^2 \mu_0 \Dv{v_\eps} : \D(W_\eps \phi) \dd x \dd t \to \int_0^\tau \int_\O \mu_0 M \vv \cdot \phi \dd x \dd t.
\end{align*}

The second contribution to the right-hand side of \eqref{dissip_decom_eq1} vanishes due to
\begin{align*}
    \int_0^\tau \int_{\O_\eps} \eps^2 \mu_0 \Dv{v_\eps} : \D B_\eps(\phi) \dd x \dd t \lesssim \eps^2 \|\Dv{v_\eps}\|_{L^2 L^2} \|\nabla B_\eps(\phi)\|_{L^2 L^2} \lesssim \eps \to 0.
\end{align*}

Finally, we show that also the third contribution to the right-hand side of \eqref{dissip_decom_eq1} vanishes, where we pay special attention to $\We_1$.
To this end, we notice that
\begin{align*}
    |g(\We_1 |\Dv{v_\eps}|)| \lesssim |\We_1 \Dv{v_\eps}|^\varkappa + |\We_1 \Dv{v_\eps}|^{p-2} \mathbf{1}_{p>2} && \forall \varkappa \in (0,1).
\end{align*}
Hence,
\begin{align*}
& \eps^2 \Big|  \int_{\O} g(|\We_1 \Dv{\vv_\eps}|) \Dv{\vv_\eps} : \Dv{w_\eps} \dd x \dd t \Big|\\
  &\lesssim \eps^2 \int_{\O} ( \We_1^\varkappa |\Dv{v_\eps} |^{1+\varkappa} + \We_1^{p-2} |\Dv{v_\eps} |^{p-1} \mathbf{1}_{p>2} ) |\Dv{w_\eps} | \dd x \dd t  \\
  &\lesssim \eps^2 ( \We_1^\varkappa \| | \nabla  \vv _{\eps} |^{1+\varkappa} \|_{L^{\frac{2}{1+\varkappa}}(\O)}  \| \vc w_{\eps} \|_{W^{1, \frac{2}{1-\varkappa}}(\O)} + \We_1^{p-2} \| | \nabla  \tilde \vv _{\eps} |^{p-1} \|_{L^{\frac{p}{p-1}}(\O)}  \| \vc w_\eps \|_{W^{1, p}(\O)} \mathbf{1}_{p>2} )  \\
  & \lesssim \eps^2 ( \We_1^\varkappa \eps^{-(1+\varkappa)} \eps^{-1} + \We_1^{p-2} (\eps^{-2} \We_1^{2-p})^\frac{p-1}{p} \eps^{-1} \mathbf{1}_{p>2} ) \\
  &=  (\We_1 / \eps)^{\varkappa} + (\We_1 / \eps)^\frac{p-2}{p} \mathbf{1}_{p>2} \to 0,
\end{align*}
where we used assumption \eqref{assCsts} in form of $\We_1 / \eps \to 0$ in the last step.
Collecting all terms, we arrive at $\langle \mu_0 M\vv - \vc f,\phi\rangle=0$.
Since the solenoidal $\phi\in C_c^\infty([0,T]\times\O;\R^3)$ was arbitrary, there is for all $t\in [0,T]$ a pressure $\pi(t)\in L^2(\O;\R)$ with zero mean such that $\mu_0 M\vv(t)=\vc f(t) -\nabla \pi(t)$ by de Rham's argument (see \cite{Temam1977}, Chapter~I, Proposition~1.1).
Since $\vv$ and $\vc f$ belong to $L^2((0,T)\times\O;\R^3)$, so does $\nabla\pi$ and hence by Poincar\'e's inequality also $\pi$.

\section{Quantitative convergence}\label{sec:REI}
Before presenting the proof of Theorem \ref{thm:quant}, we derive a relative energy inequality for the scaled viscoelastic system.
Relative energy methods are well established in the study of fluid equations and have been used successfully to obtain quantitative convergence and weak--strong uniqueness results; see, for instance, Theorem 4 in Section 3.4 of \cite{FeireislNovotny2022} for the case without polymeric stress.
The presence of the additional stress variable requires a modified relative energy functional and a corresponding extension of the underlying argument.
The inequality established below forms the key ingredient in the proofs of Theorem \ref{thm:quant} and \ref{thm:wk-str}.

\begin{lemma}[Relative energy inequality]\label{lem:REI}
    Let $D\subseteq \Omega$ be a domain with Lipschitz boundary.
    Let $\eps>0$, and assume that $(\vv_\eps,\T_\eps)$ is a weak solution to \eqref{NSE} on $D$.
    Consider arbitrary
    \begin{align}\label{regStrSol}
    \begin{cases}
        (\vu, \S) \in W^{1,\infty}(0,T; L^2(D; \R^3 \times \R^{3 \times 3})) \cap L^2(0,T; W^{2,\infty}(D; \R^3 \times \R^{3 \times 3})), \\
        \div \vu = 0, \qquad \vu |_{\del D} = 0, \qquad \del_n \S |_{\del D} = 0.
    \end{cases}
    \end{align}
    Then for the relative energy
    \begin{align*}
        E(\vv_\eps, \T_\eps | \vu, \S) = \int_{D} \eps^\lambda (1-\beta_\eps) |\vv_\eps - \vu|^2 + \frac12\eps^{\kappa+\xi} |\T_\eps - \S|^2 \dd x,
    \end{align*}
    it holds
    \begin{align}\label{REI}
    \begin{split}
    &\partial_t E(\vv_\eps, \T_\eps | \vu, \S) + \int_{D} \eps^\xi |\T_\eps - \S|^2 \dd x \\
    &\quad + \int_{D} \eps^{\kappa+\xi} \eps_1 \nabla (\T_\eps - \S) : ( \gamma(\nabla \T_\eps) \nabla \T_\eps - \gamma(\nabla \S) \nabla \S ) \dd x \\
    &\quad + \int_{D} 2 \eps^2 (1-\beta_\eps) (\Dv{\vv_\eps} - \D\vu) : ( \mu(\We_1 \Dv{\vv_\eps}) \Dv{\vv_\eps} - \mu(\We_1 \Dv{u}) \Dv{u} ) \dd x \\
    &\le 2\int_{D} (1-\beta_\eps) (\vv_\eps - \vu) \cdot (\vc f + \eps^\xi \div \S - \eps^\lambda \del_t \vu - N_\eps(\vu)) \dd x \\
    &\quad + 2\int_{D} (1-\beta_\eps) \eps^\lambda (\vv_\eps - \vu) \cdot ((\vu - \vv_\eps) \cdot \nabla) \vu \dd x \\
    &\quad - \int_{D} \eps^\xi (\T_\eps - \S) : ( \S + \eps^\kappa \oldr{\vu}{\S} - 2(1-\beta_\eps) \Dv{u} ) \dd x \\
    &\quad + \int_{D} \eps^{\kappa+\xi} (\T_\eps - \S) : [ ((\vu - \vv_\eps) \cdot \nabla) \S + \vc B(\vv_\eps - \vu, \T_\eps) + \vc B(\vu, \T_\eps - \S) ] \dd x, 
    \end{split}
    \end{align}
    where
    \begin{align*}
        N_\eps(\vu)&:=\eps^\lambda \div(\vu \otimes \vu) - \eps^2 \div(\mu(\We_1 \Dv{\vu})\Dv{\vu}), \\
        \oldr{\vu}{\S} &:= \del_t \S + (\vu \cdot \nabla) \S  - \eps_1\div(\gamma(\nabla\S)\nabla\S) - \vc B(\vu, \S).
\end{align*}
\end{lemma}
\begin{proof}
In \eqref{TotEnEps} we stated the total energy inequality for weak solutions to \eqref{NSE} on $\O_\eps$.
The same argument applies to weak solutions on $D$, which immediately gives the result for $(\vu,\S)=(0,0)$, i.e., for almost every $\tau\in [0,T]$ it holds
\begin{align}\label{TotEnEpsD}
\begin{split}
    \left[ E(\vv_\eps,\T_\eps|0,0) \right]_{t=0}^{t=\tau} 
    &+ \int_0^\tau \int_{D} 2(1-\beta_\eps) \eps^2 \mu(\We_1 \Dv{v_\eps}) |\Dv{v_\eps}|^2 \dd x \dd t  \\
    &\qquad + \int_0^\tau \int_{D} \eps^{\kappa+\xi} \eps_1 \gamma( \nabla \T_\eps) |\nabla \T_\eps|^2 + \eps^\xi |\T_\eps| ^2 \dd x \dd t \\
    &\leq \int_0^\tau \int_{D} 2(1-\beta_\eps) \vv_\eps \cdot \vc f +\vc \eps^{\kappa+\xi} B(\vv_\eps,\T_\eps):\T_\eps \dd x \dd t , 
\end{split}    
\end{align}
where we have also used \eqref{BAconversion}.
Observe that by a density argument, it is enough to establish \eqref{REI} for arbitrary smooth functions $\vu \in C_c^\infty([0,T] \times D; \R^3), \, \S \in C^\infty([0,T] \times D; \R^{3\times 3})$ with $\div \vu=0$ and $\del_n \S |_{\del D} = 0$.
We will therefore assume this regularity in the sequel.
In particular, we may use $(\vu,\S)$ as test functions in the definition of a weak solution, and obtain
       \begin{align*}
            &\left[ \int_D 2\eps^\lambda(1-\beta_\eps)\vv_\eps\cdot \vu + \eps^{\kappa+\xi} \T_\eps : \S \dd x \right]_{t=0}^{t=\tau} - 
            \int_0^\tau\int_D 2\eps^\lambda (1-\beta_\eps)\vv_\eps\cdot \del_t \vu + \eps^{\kappa+\xi} \T_\eps : \del_t \S \dd x \dd t\\
            &\qquad = \int_0^\tau\int_{D}  2(1-\beta_\eps) \vc f \cdot \vu  \dd x \dd t - \int_0^\tau \int_{D} 2(1-\beta_\eps)\big( \eps^\xi \T_\eps - n_\eps(\vv_\eps) \big) : \nabla \vu \dd x \dd t\\
            &\qquad \quad + \int_0^\tau \int_{D} ( 2\eps^\xi (1-\beta_\eps) \Dv{v_\eps}-\eps^\xi \T_\eps-\eps^{\kappa+\xi}(\vv_\eps \cdot \nabla) \T_\eps + \eps^{\kappa+\xi} \vc B(\vv_\eps,\T_\eps)) : \S \dd x \dd t \\
             &\qquad \quad - \int_0^\tau \int_{D} \eps_1 \eps^{\kappa+\xi} \gamma(\nabla \T_\eps) \nabla \T_\eps : \nabla \S \dd x \dd t,
        \end{align*}
where we have introduced the convenient short-hand notation
\begin{align*}
    n_\eps(\vv)&:=\eps^\lambda \vv \otimes \vv - \eps^2 \mu(\We_1 \Dv{\vv})\Dv{\vv}.
\end{align*}
Therefore,
\begin{align*}
    &\left[E(\vv_\eps, \T_\eps | \vu, \S) \right]_{t=0}^{t=\tau}=
    \left[ E(\vv_\eps,\T_\eps | 0,0) \right]_{t=0}^{t=\tau}  \\
    &\quad +\int_0^\tau\int_D 2\eps^\lambda (1-\beta_\eps)\vu\cdot \del_t \vu + \eps^{\kappa+\xi} \S : \del_t \S \dd x \dd t
    -\left[ \int_D 2\eps^\lambda(1-\beta_\eps)\vv_\eps\cdot \vu + \eps^{\kappa+\xi} \T_\eps : \S \dd x \right]_{t=0}^{t=\tau}
    \\ 
    &= \left[ E(\vv_\eps,\T_\eps | 0,0) \right]_{t=0}^{t=\tau} \\
    &\quad - \int_0^\tau\int_{D} 2\eps^\lambda (1-\beta_\eps) (\vv_\eps - \vu) \cdot \del_t \vu \dd x 
    - \int_{D} 2(1-\beta_\eps)  \vu \cdot \vc f \dd x \dd t \\
    &\quad + \int_0^\tau \int_{D} 2(1-\beta_\eps) \nabla \vu : (\eps^\xi \T_\eps - n_\eps(\vv_\eps)) \dd x \dd t - \int_0^\tau \int_{D} \eps^{\kappa+\xi} (\T_\eps - \S) : \del_t \S  \dd x \dd t \\
    &\quad  - \int_0^\tau \int_{D} \eps^\xi  \S : (2(1-\beta_\eps) \Dv{\vv_\eps} - \T_\eps) \dd x \dd t + \int_0^\tau \int_{D} \eps^{\kappa+\xi} \S : ( (\vv_\eps \cdot \nabla) \T_\eps - \vc B(\vv_\eps, \T_\eps)) \dd x \dd t\\
    &\quad + \int_0^\tau \int_D \eps^{\kappa+\xi} \nabla \S:(\eps_1 \gamma(\nabla \T_\eps) \nabla \T_\eps) \dd x \dd t,
\end{align*}
and hence we may use \eqref{TotEnEpsD}, $\int_{D} (\T_\eps: \Dv{\vv_\eps} - \nabla \vv_\eps :  \T_\eps) \dd x=0$, as well as
\begin{align*}
	\int_{D} \T_\eps : ( (\vv_\eps \cdot \nabla) \T_\eps) \dd x = 0
	\quad \text{and} \quad
	\int_{D} \nabla\vv_\eps : (\vv_\eps\otimes\vv_\eps) \dd x = 0
\end{align*}
to learn
\begin{align*}
    &\left[E(\vv_\eps, \T_\eps | \vu, \S) \right]_{t=0}^{t=\tau}    \\ 
    &\le \int_0^\tau\int_{D}  2(1-\beta_\eps) (\vv_\eps - \vu) \cdot (\vc f - \eps^\lambda \del_t \vu) \dd x \dd t 
    - \int_0^\tau \int_{D} 2(1-\beta_\eps) \nabla (\vv_\eps - \vu) : (\eps^\xi \T_\eps - n_\eps(\vv_\eps)) \dd x \dd t \\
    &\quad + \int_0^\tau \int_{D} \eps^\xi  (\T_\eps - \S) : (2(1-\beta_\eps) \Dv{\vv_\eps} - \T_\eps ) \dd x \dd t\\
    &\quad - \int_0^\tau \int_{D} \eps^{\xi+\kappa} (\T_\eps - \S) : ( (\vv_\eps \cdot \nabla) \T_\eps  - \vc B(\vv_\eps,\T_\eps) + \del_t \S)  \dd x \dd t\\
    &\quad - \int_0^\tau \int_D \eps^{\kappa+\xi}\nabla (\T_\eps-\S): \eps_1\gamma(\nabla\T_\eps)\nabla\T_\eps \dd x \dd t.
\end{align*}
Let us now rewrite the individual terms.
Since
\begin{align*}
 - \int_{D}  \nabla (\vv_\eps - \vu) : (\eps^\xi \S - n_\eps(\vu)) \dd x= \int_{D} (\vv_\eps - \vu) \cdot (\eps^\xi \div \S - N_\eps(\vu)) \dd x,   
\end{align*}
we have
\begin{align*}
     &\int_{D}  (\vv_\eps - \vu) \cdot (\vc f - \eps^\lambda \del_t \vu) \dd x - \int_{D} \nabla (\vv_\eps - \vu) : (\eps^\xi \T_\eps - n_\eps(\vv_\eps)) \dd x \\
     &=  \int_{D}  (\vv_\eps - \vu) \cdot (\vc f +\eps^\xi\div\S- \eps^\lambda \del_t \vu - N_\eps(\vu)) \dd x \\
     &\quad + \int_{D} \nabla (\vv_\eps - \vu) : (n_\eps(\vu)-n_\eps(\vv_\eps)) \dd x  - \int_{D}  \nabla(\vv_\eps - \vu) : \eps^\xi( \T_\eps - \S) \dd x.
\end{align*}

Next,
\begin{align*}
    &\int_{D} \eps^\xi (\T_\eps - \S) : ( 2(1-\beta_\eps) \Dv{\vv_\eps} - \T_\eps - \eps^\kappa [ (\vv_\eps \cdot \nabla) \T_\eps - \vc B(\vv_\eps, \T_\eps) + \del_t \S ] ) \dd x \\
    &= -\int_{D} \eps^\xi (\T_\eps - \S) : ( \S + \eps^\kappa \oldr{\vu}{\S} - 2(1-\beta_\eps) \Dv{u} ) \dd x \\
    &\quad + \int_{D} \eps^\xi (\T_\eps - \S) : (2(1-\beta_\eps) \Dv{\vv_\eps} - \T_\eps - \eps^\kappa [ (\vv_\eps \cdot \nabla) \T_\eps  - \vc B(\vv_\eps, \T_\eps) ] ) \dd x \\
    &\quad + \int_{D} \eps^\xi (\T_\eps - \S) : (\eps^\kappa [ (\vu \cdot \nabla) \S - \vc B(\vu, \S) ] + \S - 2(1-\beta_\eps) \Dv{u}) \dd x \\
    &\quad + \int_D \eps^{\kappa+\xi}(\T_\eps-\S) : (\eps_1\gamma(\nabla\S)\nabla\S) \dd x.
\end{align*}

Hence, using also integration by parts to see that
\begin{align*}
    \int_{D} (\vv_\eps-\vu) \cdot \div(\T_\eps - \S) + (\T_\eps - \S) : (\Dv{\vv_\eps} - \Dv{u}) \dd x = 0,
\end{align*}
we find
\begin{align*}
    &\left[E(\vv_\eps, \T_\eps | \vu, \S) \right]_{t=0}^{t=\tau} \\
    &\le \int_0^\tau \int_{D}  2(1-\beta_\eps) (\vv_\eps - \vu) \cdot (\vc f +\eps^\xi \div \S - \eps^\lambda \del_t \vu - N_\eps(\vu)) \dd x \dd t \\
    &\quad + \int_0^\tau \int_{D} 2(1-\beta_\eps) \nabla (\vv_\eps - \vu) : (n_\eps(\vu)-n_\eps(\vv_\eps)) \dd x \dd t\\
    &\quad - \int_0^\tau \int_{D} \eps^\xi (\T_\eps - \S) : ( \S + \eps^\kappa \oldr{\vu}{\S} - 2(1-\beta_\eps) \Dv{u} ) \dd x \dd t\\
    &\quad - \int_0^\tau \int_{D} \eps^\xi |\T_\eps - \S|^2 \dd x \dd t\\
    &\quad - \int_0^\tau \int_{D} \eps^{\kappa+\xi} \nabla (\T_\eps - \S) : ( \eps_1 \gamma(\nabla \T_\eps) \nabla \T_\eps - \eps_1 \gamma(\nabla \S) \nabla \S ) \dd x \dd t\\
    &\quad + \int_0^\tau \int_{D} \eps^{\kappa+\xi} (\T_\eps - \S) : ( (\vu \cdot \nabla) \S - (\vv_\eps \cdot \nabla) \T_\eps + \vc B(\vv_\eps, \T_\eps) - \vc B(\vu, \S) ) \dd x \dd t.
\end{align*}
Since
\begin{align*}
    \int_{D} \nabla (\vv_\eps-\vu) : ((\vv_\eps \otimes \vv_\eps)-(\vu\otimes\vu)) \dd x = \int_{D} (\vv_\eps - \vu) \cdot ((\vu - \vv_\eps) \cdot \nabla)\vu \dd x,
\end{align*}
where we used incompressibility of $\vv_\eps$ and $\vu$, we may recast this equivalently as
\begin{align}\label{REI_e3}
\begin{split}
    &\left[E(\vv_\eps, \T_\eps | \vu, \S) \right]_{t=0}^{t=\tau} + \int_0^\tau \int_{D} \eps^\xi |\T_\eps - \S|^2 \dd x \dd t \\
    &\quad + \int_0^\tau  \int_{D} \eps^{\kappa+\xi} \eps_1 \nabla (\T_\eps - \S) : ( \gamma(\nabla \T_\eps) \nabla \T_\eps - \gamma(\nabla \S) \nabla \S ) \dd x \dd t\\
    &\quad + \int_0^\tau  \int_{D} 2\eps^2 (1-\beta_\eps) \nabla (\vv_\eps - \vu) : ( \mu(\We_1 \Dv{\vv_\eps}) \Dv{\vv_\eps} - \mu(\We_1 \Dv{u}) \Dv{u} ) \dd x \dd t\\
    &\le \int_0^\tau  \int_{D} 2(1-\beta_\eps) (\vv_\eps - \vu) \cdot ( \vc f +\eps^\xi \div \S - \eps^\lambda \del_t \vu - N_\eps(\vu)) \dd x \dd t \\
    &\quad + \int_0^\tau \int_{D} 2(1-\beta_\eps) \eps^\lambda (\vv_\eps - \vu) \cdot ((\vu - \vv_\eps) \cdot \nabla) \vu \dd x \dd t\\
    &\quad - \int_0^\tau \int_{D} \eps^\xi (\T_\eps - \S) : ( \S + \eps^\kappa \oldr{\vu}{\S} - 2(1-\beta_\eps) \Dv{u} ) \dd x \dd t\\
    &\quad + \int_0^\tau \int_{D} \eps^{\kappa+\xi} (\T_\eps - \S) : ( (\vu \cdot \nabla) \S - (\vv_\eps \cdot \nabla) \T_\eps + \vc B(\vv_\eps, \T_\eps) - \vc B(\vu, \S) ) \dd x \dd t.
\end{split}
\end{align}
Subtracting from \eqref{REI_e3} the same estimate \eqref{REI_e3} with $\tau$ replaced by $\tau\pm h$ for small $h>0$ and dividing by $h$, we obtain the final relative energy inequality \eqref{REI} in the limit $h\to 0$, if we additionally observe that by incompressibility, the last integral can be recast as
\begin{align*}
    &\int_{D} \eps^{\kappa+\xi} (\T_\eps - \S) : ( (\vu \cdot \nabla) \S - (\vv_\eps \cdot \nabla) \T_\eps + \vc B(\vv_\eps, \T_\eps) - \vc B(\vu, \S) ) \dd x \\
    &= \int_{D} \eps^{\kappa+\xi} (\T_\eps - \S) : [ ((\vu - \vv_\eps) \cdot \nabla) \S + \vc B(\vv_\eps - \vu, \T_\eps) + \vc B(\vu, \T_\eps - \S) ] \dd x.
    \qedhere
\end{align*}
\end{proof}

\begin{proof}[Proof of Theorem \ref{thm:quant}]
The strategy is to use Lemma \ref{lem:REI} with $D:=\O_\eps$, $\vu:=\vu_\eps := W_\eps \vv - B_\eps(\vv)$, and $\S=0$, where $\vv$ is the strong solution to Darcy's law, $W_\eps$ is the function from Lemma~\ref{lem:Weps}, and $B_\eps$ is the operator from Lemma~\ref{lem:Bog-type}.
The choice $\S=0$ is motivated by the fact that the polymeric stress is not seen in the limiting system.
Furthermore, note that by the strong solution property of $\vv$, we have
\begin{align*}
    \|\vu_\eps\|_{W^{1,\infty} L^\infty} + \eps \|\nabla \vu_\eps\|_{L^\infty L^p} \lesssim 1, \qquad \forall p \in [1,\infty).
\end{align*}

The relative energy inequality \eqref{REI} yields
\begin{align*}
    &\partial_t E(\vv_\eps, \T_\eps | \vu_\eps, 0)
    + \int_{\O_\eps} \eps^\xi |\T_\eps|^2 \dd x
    + \int_{\O_\eps} \eps^{\kappa+\xi} \eps_1 \nabla \T_\eps : ( \gamma(\nabla \T_\eps) \nabla \T_\eps ) \dd x \notag \\
    &\quad + \int_{\O_\eps} 2 \eps^2 (1-\beta_\eps) (\Dv{v_\eps} - \Dv{u_\eps}) : ( \mu(\We_1 \Dv{v_\eps}) \Dv{v_\eps} - \mu(\We_1 \Dv{u_\eps}) \Dv{u_\eps} ) \dd x \notag \\
    &\le \int_{\O_\eps} 2(1-\beta_\eps) (\vv_\eps - \vu_\eps) \cdot (\vc f - \eps^\lambda \del_t \vu_\eps - N_\eps(\vu_\eps)) \dd x \notag \\
    &\quad + \int_{\O_\eps} 2(1-\beta_\eps) \eps^\lambda (\vv_\eps - \vu_\eps) \cdot ((\vu_\eps - \vv_\eps) \cdot \nabla) \vu_\eps \dd x \\
    &\quad + \int_{\O_\eps} 2(1-\beta_\eps) \eps^\xi \T_\eps : \Dv{u_\eps} + \eps^{\kappa+\xi} \T_\eps : \vc B(\vv_\eps, \T_\eps) \dd x .
\end{align*}

Then, we estimate
\begin{align*}
    &\int_{\O_\eps} (1-\beta_\eps) (\vv_\eps - \vu_\eps) \cdot \eps^\lambda \del_t \vu_\eps \dd x \lesssim \eps^\lambda (1-\beta_\eps) \|\vv_\eps - \vu_\eps\|_{L^2} \|\vu_\eps\|_{W^{1,\infty} L^2} \lesssim E(\vv_\eps, \T_\eps | \vu_\eps, 0) + \eps^\lambda, \\
    &\int_{\O_\eps} 2(1-\beta_\eps) \eps^\xi \T_\eps : \Dv{u_\eps} \dd x \lesssim \eps^\xi \|\vu_\eps\|_{L^2} \|\nabla \T_\eps\|_{L^q} \lesssim \eps^\frac{q \xi}{q-1} [\eps^{\kappa+\xi} \eps_1]^{-\frac{1}{q-1}} + \delta \eps^{\kappa+\xi} \eps_1 \|\nabla \T_\eps\|_{L^q}^q \\
    &\qquad = [ \eps^{\xi(q-1) - \kappa} / \eps_1 ]^{\frac{1}{q-1}} + \delta \eps^{\kappa+\xi} \eps_1 \|\nabla \T_\eps\|_{L^q}^q .
\end{align*}

Next, for the contribution $\eps^\lambda\div(\vu_\eps \otimes \vu_\eps)$ to $N_\eps(\vu_\eps)$, we obtain
\begin{align*}
    &\int_{\O_\eps} (1-\beta_\eps) \eps^\lambda (\vv_\eps - \vu_\eps) \cdot \div(\vu_\eps \otimes \vu_\eps) \dd x \lesssim \eps^\lambda (1-\beta_\eps)\|\vv_\eps-\vu_\eps\|_{L^2} \|\vu_\eps\|_{L^\infty} \|\nabla \vu_\eps\|_{L^2} \\
    &\lesssim \eps^\lambda (1-\beta_\eps) \|\nabla(\vv_\eps-\vu_\eps)\|_{L^2} \lesssim \delta (1-\beta_\eps) \eps^2 \|\Dv{v_\eps} - \Dv{u_\eps}\|_{L^2}^2 + \eps^{2(\lambda-1)},
\end{align*}
and the former term can be absorbed by dissipation for $\delta > 0$ small enough. Similarly, we find
\begin{align*}
    &\int_{\O_\eps} (1-\beta_\eps) \eps^\lambda (\vv_\eps - \vu_\eps) \cdot ((\vu_\eps - \vv_\eps) \cdot \nabla) \vu_\eps \dd x \\
    &\lesssim (1-\beta_\eps) \eps^{\lambda} \|\vv_\eps - \vu_\eps\|_{L^2}^2 \|\nabla (W_\eps \vv)\|_{L^{\infty}} + (1-\beta_\eps) \eps^\lambda \|\vv_\eps - \vu_\eps\|_{L^2} \|\vv_\eps - \vu_\eps\|_{L^6} \|\nabla B_\eps(\vv)\|_{L^3} \\
    &\lesssim (1-\beta_\eps) \eps^{\lambda+1} \|\nabla(\vv_\eps - \vu_\eps)\|_{L^2}^2 
    \lesssim \delta (1-\beta_\eps) \eps^2 \|\nabla (\vv_\eps - \vu_\eps)\|_{L^2}^2.
\end{align*}

Since $\vv$ is a strong solution to Darcy's law \eqref{Darcy}, it especially fulfills it pointwise in $\O_\eps$. Hence, we can replace the force and use solenoidality of $(\vv_\eps - \vu_\eps)$ to infer
\begin{align}\label{REI_e1}
\begin{split}
    &\del_t E(\vv_\eps, \T_\eps | \vu_\eps, 0) + \int_{\O_\eps} \eps^\xi |\T_\eps|^2 \dd x + \int_{\O_\eps} \eps^{\kappa+\xi} \eps_1 \gamma(\nabla \T_\eps) |\nabla \T_\eps|^2 \dd x \\
    &\quad + \int_{\O_\eps} 2 \eps^2 (1-\beta_\eps) (\Dv{v_\eps} - \Dv{u_\eps}) : ( \mu(\We_1 \Dv{v_\eps}) \Dv{v_\eps} - \mu(\We_1 \Dv{u_\eps}) \Dv{u_\eps} ) \dd x \\
    &\lesssim - \int_{\O_\eps} (1-\beta_\eps) \eps^2 \nabla (\vv_\eps - \vu_\eps) : \mu(\We_1 \Dv{u_\eps}) \Dv{u_\eps} \dd x \\
    &\quad + \int_{\O_\eps} (1-\beta_\eps) (\vv_\eps - \vu_\eps) \cdot \mu_0 M \vv \dd x + \int_{\O_\eps} \eps^{\kappa+\xi} \T_\eps : \vc B(\vv_\eps, \T_\eps) \dd x + E(\vv_\eps, \T_\eps | \vu_\eps, 0) \\
    &\quad + \eps^{\lambda}+ \eps^{2(\lambda-1)} + [\eps^{\xi(q-1)-\kappa} / \eps_1]^\frac{1}{q-1}.
\end{split}
\end{align}

Let us rewrite the first integral on the right-hand side using the definition of $\mu$ from \eqref{assMu} and $\vu_\eps$:
\begin{align*}
    &- \int_{\O_\eps} (1-\beta_\eps) \eps^2 \nabla (\vv_\eps - \vu_\eps) : \mu(\We_1 \Dv{u_\eps}) \Dv{u_\eps} \dd x \\
    &= - \int_{\O_\eps} (1-\beta_\eps) \eps^2 \nabla (\vv_\eps - \vu_\eps) : \mu_0  \Dv{u_\eps} \dd x - \int_{\O_\eps} (1-\beta_\eps) \eps^2 \nabla (\vv_\eps - \vu_\eps) : g(\We_1 |\Dv{u_\eps}|) \Dv{u_\eps} \dd x \\
    &= - \int_{\O_\eps} (1-\beta_\eps) \eps^2 \nabla (\vv_\eps - \vu_\eps) : \mu_0  \D(W_\eps \vv) \dd x + \int_{\O_\eps} (1-\beta_\eps) \eps^2 \nabla (\vv_\eps - \vu_\eps) : \mu_0  \D B_\eps(\vv) \dd x \\
    &\quad - \int_{\O_\eps} (1-\beta_\eps) \eps^2 \nabla (\vv_\eps - \vu_\eps) : g(\We_1 |\Dv{u_\eps}|) \Dv{u_\eps} \dd x .
\end{align*}

The second integral is estimated as
\begin{align*}
    &\int_{\O_\eps} (1-\beta_\eps) \eps^2 \nabla (\vv_\eps - \vu_\eps) : \mu_0  \D B_\eps(\vv) \dd x \lesssim \eps^2 (1-\beta_\eps) \|\nabla (\vv_\eps - \vu_\eps)\|_{L^2} \|\nabla B_\eps(\vv)\|_{L^2} \\
    &\lesssim \delta \eps^2 (1-\beta_\eps) \|\nabla (\vv_\eps - \vu_\eps)\|_{L^2}^2 + \eps^2\|(W_\eps - \mathbb{I}) : \nabla\vv\|_{L^2}^2 \\
    &\lesssim \delta \eps^2 (1-\beta_\eps) \|\nabla (\vv_\eps - \vu_\eps)\|_{L^2}^2 + \eps^2\|W_\eps - \mathbb{I}\|_{[W^{1,1}]'}^2 \|\vv\|_{W^{2,1}}^2 \\
    &\lesssim \delta \eps^2 (1-\beta_\eps) \|\nabla (\vv_\eps - \vu_\eps)\|_{L^2}^2 + \eps^4 .
\end{align*}

As for the first one,
\begin{align*}
    &- \int_{\O_\eps} (1-\beta_\eps) \eps^2 \nabla (\vv_\eps - \vu_\eps) : \mu_0  \D(W_\eps \vv) \dd x \\
    &= - (1-\beta_\eps) \mu_0 \langle \eps^2 (-\Delta W_\eps + \nabla Q_\eps), \vv \otimes (\vv_\eps - \vu_\eps) \rangle \\
    &\quad + \int_{\O_\eps} (1-\beta_\eps) \eps^2 (\vv_\eps - \vu_\eps) \cdot \mu_0  (\Delta(W_\eps \vv) - \Delta W_\eps \vv - (Q_\eps \cdot \nabla) \vv) \dd x .
\end{align*}
We see that, by definition of $(W_\eps, Q_\eps)$ from Lemma~\ref{lem:Weps}, the first term above cancels with the term
\begin{align*}
    \int_{\O_\eps} (1-\beta_\eps)(\vv_\eps-\vu_\eps)\cdot \mu_0 M\vv \dd x
\end{align*}
on the right-hand side of \eqref{REI_e1}.
As for the second term, we estimate
\begin{align*}
    &\int_{\O_\eps} (1-\beta_\eps) \eps^2 (\vv_\eps - \vu_\eps) \cdot \mu_0  (\Delta(W_\eps \vv) - \Delta W_\eps \vv - (Q_\eps \cdot \nabla) \vv) \dd x \\
    &\lesssim (1-\beta_\eps) \eps^2 \|\vv_\eps - \vu_\eps\|_{L^2} \|\Delta (W_\eps \vv) - \Delta W_\eps \vv - (Q_\eps \cdot \nabla) \vv\|_{L^2} \\
    &\lesssim \delta (1-\beta_\eps) \eps^2 \|\nabla (\vv_\eps - \vu_\eps)\|_{L^2}^2 + \eps^2.
\end{align*}

Lastly, we see
\begin{align*}
    |g(s)| \lesssim s + s^{p-2} \mathbf{1}_{p>3} \quad \forall s \geq 0
\end{align*}
such that
 \begin{align*}
     &\int_{\O_\eps} (1-\beta_\eps) \eps^2 \nabla (\vv_\eps - \vu_\eps) : g(\We_1 |\Dv{u_\eps}|) \Dv{u_\eps} \dd x \\
     &\lesssim \int_{\O_\eps} (1-\beta_\eps) \eps^2 | \nabla (\vv_\eps - \vu_\eps) | (|\We_1 \Dv{u_\eps}| + |\We_1 \Dv{u_\eps}|^{p-2} \mathbf{1}_{p>3}) |\Dv{u_\eps}| \dd x \\
     &\lesssim (1-\beta_\eps) \eps^2 \|\nabla (\vv_\eps - \vu_\eps)\|_{L^2} (\We_1 \|\nabla \vu_\eps\|_{L^4}^2 + \We_1^{p-2} \|\nabla \vu_\eps\|_{L^{2(p-1)}}^{p-1} \mathbf{1}_{p>3}) \\
     &\lesssim \delta \eps^2 (1-\beta_\eps) \|\nabla (\vv_\eps - \vu_\eps)\|_{L^2}^2 + (\We_1 / \eps)^2 + (\We_1 / \eps)^{2(p-2)} \mathbf{1}_{p>3} \\
     &\lesssim \delta \eps^2 (1-\beta_\eps) \|\nabla (\vv_\eps - \vu_\eps)\|_{L^2}^2 + (\We_1/\eps)^2.
 \end{align*}

 Therefore, estimate \eqref{REI_e1} takes the form
 \begin{align*}
     &\del_t E(\vv_\eps, \T_\eps | \vu_\eps, 0) + \int_{\O_\eps} \eps^\xi |\T_\eps|^2 \dd x + \int_{\O_\eps} \eps^{\kappa+\xi} \eps_1 \gamma(\nabla \T_\eps) |\nabla \T_\eps|^2 \dd x \notag \\
    &\quad + \int_{\O_\eps} 2 \eps^2 (1-\beta_\eps) (\Dv{v_\eps} - \Dv{u_\eps}) : ( \mu(\We_1 \Dv{v_\eps}) \Dv{v_\eps} - \mu(\We_1 \Dv{u_\eps}) \Dv{u_\eps} ) \dd x \notag \\
    &\lesssim \int_{\O_\eps} \eps^{\kappa+\xi} \T_\eps : \vc B(\vv_\eps, \T_\eps) \dd x + E(\vv_\eps, \T_\eps | \vu_\eps, 0) + \eps^{2(\lambda-1)} + \eps^2 + [\eps^{\xi(q-1)-\kappa} / \eps_1]^\frac{1}{q-1} + (\We_1/ \eps)^2,
 \end{align*}
 where we have also used $\lambda\ge \min\{2,2(\lambda-1)\}$ in order to estimate $\eps^\lambda\le \eps^2+\eps^{2(\lambda-1)}$.

 Finally, we estimate similarly to before
 \begin{align}\label{REI_e2}
 \begin{split}
     &\int_{\O_\eps} \eps^{\kappa+\xi} \T_\eps : \vc B(\vv_\eps, \T_\eps) \dd x \lesssim \int_{\O_\eps} \eps^{\kappa+\xi} |\vv_\eps| |\T_\eps| |\nabla \T_\eps| \dd x \\
     &\lesssim \eps^{\kappa+\xi+1} \|\Dv{v_\eps}\|_{L^2} \|\nabla \T_\eps\|_{L^q} + \eps^{\kappa+\xi+1} \|\Dv{v_\eps}\|_{L^2} \|\nabla \T_\eps\|_{L^q}^2 + \eps^{\kappa+\xi+2} \|\Dv{v_\eps}\|_{L^2}^2 \|\nabla \T_\eps\|_{L^q}^2 \\
     &\lesssim \eps^{\kappa+\xi} \|\nabla \T_\eps\|_{L^q} + \eps^{\kappa+\xi} \|\nabla \T_\eps\|_{L^q}^2 \\
     &\lesssim \eps^{\kappa+\xi} ( \eps_1^{-\frac{1}{q-1}} + \eps_1^{-\frac{2}{q-2}} ) + \delta \eps^{\kappa+\xi} \eps_1 \|\nabla \T_\eps\|_{L^q}^q,
 \end{split}
 \end{align}
 so that \eqref{REI_e1} takes the final form
 \begin{align*}
     &\del_t E(\vv_\eps, \T_\eps | \vu_\eps, 0) + \int_{\O_\eps} \eps^\xi |\T_\eps|^2 \dd x + \int_{\O_\eps} \eps^{\kappa+\xi} \eps_1 \gamma(\nabla \T_\eps) |\nabla \T_\eps|^2 \dd x \notag \\
    &\quad + \int_{\O_\eps} 2 \eps^2 (1-\beta_\eps) (\Dv{v_\eps} - \Dv{u_\eps}) : ( \mu(\We_1 \Dv{v_\eps}) \Dv{v_\eps} - \mu(\We_1 \Dv{u_\eps}) \Dv{u_\eps} ) \dd x \notag \\
    &\lesssim E(\vv_\eps, \T_\eps | \vu_\eps, 0) + \eps^{2(\lambda-1)} + \eps^2 + [\eps^{\xi(q-1)-\kappa} / \eps_1]^\frac{1}{q-1} + [\eps^{(q-1)(\kappa+\xi)} / \eps_1]^{\frac{1}{q-1}} \\
    &\qquad + [\eps^{(q-2)(\kappa+\xi)/2} / \eps_1]^{\frac{2}{q-2}} + (\We_1 / \eps)^2.
 \end{align*}

 Gr\"onwall's inequality then yields for a.e.~$t \in [0,T]$
 \begin{align*}
     E(\vv_\eps, \T_\eps | \vu_\eps, 0)(t) &\lesssim E(\vv_{\eps 0}, \T_{\eps 0} | \vu_{\eps 0}, 0) + \eps^{2(\lambda-1)} + \eps^2 + [\eps^{\xi(q-1)-\kappa} / \eps_1]^\frac{1}{q-1} \\
    &\qquad + [\eps^{(q-1)(\kappa+\xi)} / \eps_1]^{\frac{1}{q-1}} + [\eps^{(q-2)(\kappa+\xi)/2} / \eps_1]^{\frac{2}{q-2}} + (\We_1 / \eps)^2.
 \end{align*}

 Finally, thanks to the estimates on $B_\eps$ from Lemma~\ref{lem:Bog-type} and Poincar\'e's inequality in Lemma~\ref{lem:Poinc}, it is easy to see that
 \begin{align*}
     \|\tilde \vv_\eps - \vv\|_{L^2 L^2}^2 &\lesssim \eps^2 \|\nabla (\tilde \vv_\eps - \tilde \vu_\eps)\|_{L^2 L^2}^2 + \eps^2, \\
     E(\vv_{\eps 0} , \T_{\eps 0} | \vu_{\eps 0}, 0) &\lesssim \eps^\lambda \|\tilde \vv_{\eps 0} - \vv(0,\cdot)\|_{L^2 L^2}^2 + \eps^{\kappa+\xi} \|\T_{\eps 0}\|_{L^2}^2 + \eps^2,
 \end{align*}
 finishing the proof of Theorem~\ref{thm:quant}.
\end{proof}
 \begin{rem}
     Instead of estimating $\Dv{v_\eps}$ in \eqref{REI_e2} directly, one may use
     \begin{align*}
         \|\Dv{v_\eps}\|_{L^2} \leq \|\nabla (\vv_\eps - \vu_\eps)\|_{L^2} + \|\nabla \vu_\eps\|_{L^2}
    \end{align*}
    and absorb the first part by dissipation, however, this will not give a better convergence rate since $\nabla \vu_\eps$ and $\nabla \vv_\eps$ have a similar behavior in terms of $\eps$:
    \begin{align*}
        \|\nabla \vu_\eps\|_{L^2} + \|\nabla \vv_\eps\|_{L^2} \lesssim \eps^{-1}.
    \end{align*}
 \end{rem}

\section{weak--strong uniqueness}\label{sec:wk-str}
Using the relative energy identity \eqref{REI}, we finally establish weak--strong uniqueness for system \eqref{NSE}.
In contrast to the quantitative homogenization argument, the first and the third term on the right-hand side of \eqref{REI} vanish identically in this setting, so that only the remaining second and fourth term need to be controlled.
\begin{proof}[Proof of Theorem \ref{thm:wk-str}]
    Since $(\vu, \S)$ is a strong solution, we infer from \eqref{REI} that
    \begin{align*}
        &\del_t E(\vv, \T | \vu, \S) + \int_{\O} |\T - \S|^2 \dd x + \int_{\O} \eps_1 \nabla (\T - \S) : ( \gamma(\nabla \T) \nabla \T - \gamma(\nabla \S) \nabla \S ) \dd x \\
    &\quad + \int_{\O} 2 (\Dv{v} - \D\vu) : ( \mu(\Dv{v}) \Dv{v} - \mu(\Dv{u}) \Dv{u} ) \dd x \\
    &\le 2 \int_{\O} (\vv - \vu) \cdot ((\vu - \vv) \cdot \nabla) \vu \dd x \\
    &\quad + \int_{\O} (\T - \S) : [ ((\vu - \vv) \cdot \nabla) \S + \vc B(\vv - \vu, \T) + \vc B(\vu, \T - \S) ] \dd x .
    \end{align*}
    Similar to before, we estimate
    \begin{align*}
        \int_{\O} (\vv - \vu) \cdot ((\vu - \vv) \cdot \nabla) \vu \dd x &\lesssim \|\vv - \vu\|_{L^2}^2, \\
        \int_{\O} (\T - \S) : ((\vu - \vv) \cdot \nabla) \S \dd x &\lesssim \|\T - \S\|_{L^2}^2 + \|\vv - \vu\|_{L^2}^2,  \\
        \int_{\O} (\T - \S) : \vc B(\vu, \T - \S) \dd x &\lesssim \|\T - \S\|_{L^2}^2 , \\
        \int_\O (\T - \S) : \vc B(\vv - \vu, \T) \dd x &\lesssim \|\T - \S\|_{L^2} \|\Dv{v-u}\|_{L^2} \|\T\|_{L^4} \\
        &\lesssim \|\T - \S\|_{L^2}^2 + \delta \|\Dv{v-u}\|_{L^2}^2,
    \end{align*}
    where $\delta>0$ is chosen sufficiently small to absorb this term to the left-hand side. All in all, we find, using also coercivity of the functions $\mu$ and $\gamma$ from \eqref{assGamma}--\eqref{assMu},
    \begin{align*}
        \del_t E(\vv, \T | \vu, \S) + \int_{\O} |\T - \S|^2 \dd x + \int_{\O} \eps_1 |\nabla (\T - \S)|^q \dd x + \int_{\O} |\Dv{v} - \Dv{u}|^2 \dd x \lesssim E(\vv, \T | \vu, \S),
    \end{align*}
    so that the result follows from Gr\"onwall's inequality.
\end{proof}

\section*{Acknowledgments}
{\it J.S.~gratefully acknowledges the hospitality of Charles University during a research visit to F.O.~in spring 2026. The visit was supported by F.O.'s grant PRIMUS 26/SCI/026.}\\

\paragraph{\bf Data availability} No data are available.
\paragraph{\bf Conflict of interest} The authors declare that there is no conflict of interest.

\appendix

\section{Non-dimensionalization}\label{sec:non-dim}

To arrive at the non-dimensional system \eqref{NSE}, we recall its dimensional form from \cite{KremlPokornySalom2014} as
\begin{align*}
    \begin{cases}
        \div \vv = 0 & \text{in } (0,T) \times \O_\eps,\\
        \rho \del_t \vv + \rho \div(\vv \otimes \vv) + \nabla \pi - \div(\mu(\lambda_1 \Dv{\vv}) \Dv{v}) = \rho \vc f + \div \T & \text{in } (0,T) \times \O_\eps,\\
        \T + \zeta \overset{\triangledown}{\T} = 2 \eta \Dv{v} & \text{in } (0,T) \times \O_\eps,
    \end{cases}
\end{align*}
where $\rho >0$ is constant, $\Dv{v}=\frac12 (\nabla \vv + \nabla^T \vv)$, and ${\mathbf B}(\vv,\T) = {\mathbf W}\T-\T {\mathbf W} + a(\Dv{v} \T + \T \Dv{v})$ with ${\mathbf W}=\frac12 (\nabla \vv - \nabla^T \vv)$. Moreover, $\zeta, \lambda_1$ are time parameters appearing in the upper-convected derivative $\T + \zeta \overset{\triangledown}{\T} = 2 \eta \Dv{v}$, where the Oldroyd derivative is
\begin{align*}
    \overset{\triangledown}{\T} = \partial_t \T + (\vv \cdot \nabla) \T - \eps_1 \div(\gamma(\lambda_2 \nabla\T)\nabla\T) - \vc B(\vv, \T).
\end{align*}
The parameter $\lambda_2$ has the unit of $1/(\text{force per volume})$ and is needed for dimensional reasons only. Now, we introduce characteristic values of length $L_c$, velocity $U_c$, time $T_c$, pressure $P_c$, force $F_c$, dissipation $\gamma_c$, and viscosities $\mu_0, \eta$, and define the Strouhal, Euler, (total) Reynolds, Froude, and Weissenberg number by
\begin{align*}
    \Sr = \frac{L_c}{T_c U_c}, && \Eu = \frac{P_c}{\rho U_c^2}, && \Rey = \frac{\rho U_c L_c}{\mu_0 + \eta}, && \Fr = \frac{U_c}{\sqrt{F_c L_c}}, && \We = \frac{\zeta U_c}{L_c}.
\end{align*}
Note also that $\We \Sr = \zeta / T_c = {\rm De}$ is the Deborah number. Then, setting the (total) stress scale $\tau_c = (\mu_0 + \eta) U_c/L_c = \rho U_c^2 / \Rey$ and the viscosity ratio $\beta = \mu_0 / (\mu_0 + \eta)$, defining $X'=X / X_c$ for each $X \in \{\vv, \T, \pi, \vc f\}$, and dropping primes, we arrive at
\begin{align*}
    \begin{cases}
        \div \vv = 0 & \text{in } (0,T) \times \O_\eps,\\
        \Sr \del_t \vv + \div(\vv \otimes \vv) + \Eu \nabla \pi - \frac{\beta}{\Rey} \div(\mu(\We_1\Dv{v}) \Dv{v}) = \frac{1}{\Fr^2} \vc f + \frac{1}{\Rey} \div \T & \text{in } (0,T) \times \O_\eps,\\
        \T + \We [ \Sr \del_t \T + (\vv \cdot \nabla) \T - \eps_1 \frac{\gamma_c}{L_c U_c} \div( \gamma( \Xi \nabla \T) \nabla \T) - \vc B(\vv, \T) ] = 2 (1-\beta) \Dv{v} & \text{in } (0,T) \times \O_\eps,
    \end{cases}
\end{align*}
where we set $\We_1 = \lambda_1 U_c/L_c$ and $\Xi = \lambda_2 \rho U_c^2/ (\Rey L_c)$. Note also that $\rho U_c^2 / (\Rey L_c)$ has exactly the dimension of the convective term $\rho \div(\vv \otimes \vv)$, which is (force per volume), hence, $\lambda_2$ is needed. Incorporating further $\gamma_c/(L_c U_c)$ into $\eps_1$, we may write the system as
\begin{align*}
    \begin{cases}
        \div \vv = 0 & \text{in } (0,T) \times \O_\eps,\\
        \Sr \del_t \vv + \div(\vv \otimes \vv) + \Eu \nabla \pi - \frac{\beta}{\Rey} \div(\mu(\We_1 \Dv{v}) \Dv{v}) = \frac{1}{\Fr^2} \vc f + \frac{1}{\Rey} \div \T & \text{in } (0,T) \times \O_\eps,\\
        \T + \We [ \Sr \del_t \T + (\vv \cdot \nabla) \T - \eps_1 \div( \gamma( \Xi \nabla \T) \nabla \T) - \vc B(\vv, \T) ] = 2 (1-\beta) \Dv{v} & \text{in } (0,T) \times \O_\eps.
    \end{cases}
\end{align*}
Choosing proper values of the characteristic numbers in terms of $\eps$ leads to system \eqref{NSE}.\\

Furthermore, from the proof of Theorem~1.1 given in \cite{KremlPokornySalom2014}, one can extract the following energy inequalities: for almost every $\tau \in [0,T]$, we have
\begin{align}
    &\left[ \int_{\O_\eps} \Rey \Sr \frac{1}{2} |\vv|^2 \dd x \right]_{t=0}^{t=\tau} + \int_0^\tau \int_{\O_\eps} \beta \mu(\We_1 \Dv{v}) |\Dv{v}|^2 \dd x \dd t \notag \\
    &\qquad \leq \int_0^\tau \int_{\O_\eps} \frac{\Rey}{\Fr^2} \vv \cdot \vc f - \Dv{v} : \T \dd x \dd t, \label{estV} \\
    &\left[ \int_{\O_\eps} \We \Sr \frac{1}{2} |\T|^2 \dd x \right]_{t=0}^{t=\tau} + \int_0^\tau \int_{\O_\eps} \We \eps_1 \gamma(\Xi \nabla \T)|\nabla \T|^2 + |\T|^2 \dd x \dd t \notag \\
    &\qquad = \int_0^\tau \int_{\O_\eps} 2 (1-\beta) \Dv{v} : \T - 2a \We \vv \cdot \div \T^2 \dd x \dd t . \label{estT}
\end{align}
Dividing \eqref{estT} by $(1-\beta)$ and adding to \eqref{estV} times $2$, we obtain the total energy inequality
\begin{align}
    &\left[ \Sr \int_{\O_\eps} \Rey |\vv|^2 + \frac{\We}{2(1-\beta)} |\T|^2 \dd x \right]_{t=0}^{t=\tau} + \int_0^\tau \int_{\O_\eps} 2\beta \mu(\We_1\Dv{v}) |\Dv{v}|^2 \dd x \dd t \notag \\
    &\qquad + \frac{1}{1-\beta}\int_0^\tau \int_{\O_\eps} \We \eps_1 \gamma(\Xi \nabla \T) |\nabla \T|^2 + |\T| ^2 \dd x \dd t \notag \\
    &\leq 2\int_0^\tau \int_{\O_\eps} \frac{\Rey}{\Fr^2} \vv \cdot \vc f - a \frac{\We}{1-\beta} \vv \cdot \div \T^2 \dd x \dd t . \label{TotalEn}
\end{align}

Let us now explain how the scaling \eqref{Strouhaletc} is obtained. Repeating the steps done in Section~\ref{sec:unifEst} for the \emph{dimensional} system \eqref{NSE_not_scaled}, one can find the following bounds for the involved functions, provided $\vv_{\eps 0}, \T_{\eps 0}$ are uniformly bounded in $L^2(\O_\eps)$:

\begin{align*}
    \|\vv\|_{L^\infty L^2}^2 + \|\T\|_{L^\infty L^2}^2 + \|\Dv{v}\|_{L^2 L^2}^2 + \|\Dv{v}\|_{L^p L^p}^p \mathbf{1}_{p>2} + \eps_1 \|\nabla \T\|_{L^q L^q}^q + \|\T\|_{L^2 L^2}^2 \lesssim \eps^2.
\end{align*}
Thus, one may rescale $\vv \sim \eps^2$ and $\T \sim \eps$ to get the system in dimensionless variables, however, this would correspond to specific choices of $\Sr, \Rey, \Fr, \We, \Xi,$ and $\We_1$ such that \emph{all} exponents occurring in \eqref{NSE} are connected to $\eps^2$.
In contrast, we allow the characteristic numbers to vary \emph{independently} of each other.
Hence, we introduce $\lambda, \kappa, \xi, \Xi_\eps, \We_1$ to arrive at system \eqref{NSE}, with the characteristic numbers specified in \eqref{Strouhaletc}.


\end{document}